\documentclass[11pt]{article}
\usepackage{amsmath,amssymb,amsthm,amsfonts,verbatim}
\usepackage{psfrag}
\usepackage{graphicx}
\usepackage{enumerate}
\usepackage{microtype}
\usepackage[all,2cell]{xy}
\usepackage{url}
\CompileMatrices

\title{Parametrized Abel--Jacobi maps\\ and abelian cycles in the Torelli group}

\author{Thomas Church and Benson Farb
\thanks{The second author gratefully acknowledges support from the National
Science Foundation.}}

\theoremstyle{plain}
\newtheorem{theorem}{Theorem}[section]

\newtheorem{proposition}[theorem]{Proposition}

\newtheorem{conjecture}[theorem]{Conjecture}

\newtheorem*{theorem:pc}{Theorem \ref{theorem:principal congruence}}
\newtheorem*{theorem:jk}{Theorem \ref{theorem:johnson}}
\newtheorem*{theorem:br}{Theorem \ref{theorem:brunnian}}
\theoremstyle{definition}
\newtheorem{corollary}[theorem]{Corollary}
\newtheorem{definition}[theorem]{Definition}

\newcommand{\nc}{\newcommand}
\nc{\dmo}{\DeclareMathOperator}

\nc{\I}{\mathcal{I}}
\nc{\K}{\mathcal{K}}
\nc{\N}{\mathcal{N}}
\nc{\U}{\mathcal{U}}
\nc{\Q}{\mathbb{Q}}
\nc{\R}{\mathbb{R}}
\nc{\Z}{\mathbb{Z}}
\nc{\C}{\mathbb{C}}
\dmo{\GL}{GL}
\dmo{\PSL}{PSL}
\dmo{\Teich}{Teich}
\dmo{\spec}{spec}
\nc{\gin}{i}
\nc{\ga}{\Gamma}
\dmo{\Out}{Out}
\dmo{\brun}{Brun}

\dmo\im{im}
\dmo\id{id}
\dmo\Sp{Sp}
\dmo\Mod{Mod}
\dmo\PMod{PMod}
\dmo\genus{genus}
\dmo\Jac{Jac}
\dmo\fd{fd}

\dmo\Tor{{\cal T}}
\nc{\bwedge}{\textstyle{\bigwedge}}

\nc{\figwedge}{1}
\nc{\figtau}{2}
\nc{\figpants}{3}
\nc{\figexamples}{4}
\nc{\figshriek}{5}
\nc{\figabeliansurj}{6}
\nc{\figshrieksurj}{7}
\nc{\figring}{8}
\nc{\figdetect}{9}

\def\Ebar{\overline{E}}
\renewcommand{\epsilon}{\varepsilon}
\def\tauJ{\tau_{J}}
\nc{\coloneq}{\mathrel{\mathop:}\mkern-1.2mu=}

\nc{\margin}[1]{\marginpar{\scriptsize #1}}

\nc{\para}[1]{\medskip\noindent\textbf{#1.}}

\begin{document}
\maketitle
\begin{abstract}
  Let $\I_{g,\ast}$ denote the {\em Torelli group} of the genus
  $g\geq 2$ surface $S_g$ with one marked point.    This
  is the group of homotopy classes (rel basepoint) of homeomorphisms of $S_g$ fixing the basepoint and acting trivially on  $H\coloneq H_1(S_g,\Q)$.  In 1983 Johnson constructed a beautiful
  family of invariants
  \[\tau_i\colon H_i(\I_{g,\ast},\Q)\to \bwedge^{i+2}H\] 
  for $0\leq i\leq 2g-2$, using a kind of Abel--Jacobi map for
  families.  He used these invariants to detect nontrivial cycles in $\I_{g,\ast}$.
  Johnson proved that $\tau_1$ is an isomorphism, and asked
  if the same is true for $\tau_i$ with $i>1$.
  
  The goal of this paper is to introduce various methods for computing
  $\tau_i$; in particular we prove that $\tau_i$ is not injective 
  for any $2\leq i<g$, answering Johnson's question in the negative.  We also show that $\tau_2$ is
  surjective.  For $g\geq 3$, we find many classes in the image of
  $\tau_i$ and use them to deduce that $H_i(\I_{g,\ast},\Q)\neq 0$ for each $1\leq
  i< g$.  This is in contrast with the case of mapping class groups.  Many of our classes are
  stable, so we can deduce that $H_i(\I_{\infty,1},\Q)$ is
  infinite-dimensional for each $i\geq 1$.  Finally, we conjecture a
  new kind of ``representation-theoretic stability'' for the homology
  of the Torelli group, for which our results provide evidence.
\end{abstract}

\maketitle

\section{Introduction}
Let $S_g$ be a connected, closed, oriented surface of genus $g\geq 2$,
let $H\coloneq H_1(S_g,\Q)$, and let $H_\Z\coloneq H_1(S_g,\Z)$.  The
\emph{(pointed) Torelli group} $\I_{g,\ast}$ is the group of pointed
homotopy classes of pointed homeomorphisms of $S_g$ acting trivially on $H_\Z$.
Understanding $\I_{g,\ast}$, and particularly its (co)homology, is an
important problem in topology and algebraic geometry (see, e.g.,
\cite{Jo}, \cite{Ha2} and \cite{Mo1} for discussions).

\para{Parametrized Abel--Jacobi maps}
In \cite{Jo}, Johnson produced a beautiful family of invariants, and used them to 
prove that certain cycles in $H_{\ast}(\I_{g,\ast},\Q)$ are nontrivial.  He did this by
constructing, for each $0\leq i\leq 2g-2$, an
$\Sp(2g,\Z)$--equivariant homomorphism \[\tau_i\colon
H_i(\I_{g,\ast},\Q)\to \bwedge^{i+2}H.\] The maps $\tau_i$ can be
described using a kind of parametrized Abel--Jacobi map, as follows. We
first outline an explicit construction of $\tau_1\colon
H_1(\I_{g,\ast},\Q)\to \bwedge^3 H$. Let $S_{g,\ast}$ denote $S_g$
with a marked point. For any $f\in \I_{g,\ast}$, the image
$\tau_1([f])\in \bwedge^3 H$ can be computed as follows. Construct the
mapping torus \[M_f\coloneq S_{g,\ast}\times [0,1]/(f(p),0)\sim
(p,1).\] This 3--manifold is naturally a bundle with section $S_{g,\ast}\to M_f\to
S^1$, and the fact that $f$ acts trivially on $H_1(S_{g,\ast})$
implies that $H_1(M_f)$ is canonically isomorphic to
$H_1(S_{g,\ast})\times H_1(S^1)$.

The Jacobian of a Riemann surface $S_g$ is the complex torus
$\Jac(S_g)\approx T^{2g}$\linebreak given by $(H^1(S_g,\C))^*/H_1(S_g,\Z)$.  The
Abel--Jacobi map is a holomorphic map\linebreak $j\colon (S_g,\ast) \to
(\Jac(S_g),0)$ unique in its homotopy class, with the property that
the induced map on fundamental group $j_*\colon \pi_1(S_g,\ast)\to
\pi_1(\Jac(S_{g,\ast}),0)\approx H_\Z$ is the abelianization. The
composition
\[\pi_1(M_f)\to H_1(M_f)\approx H_1(S_{g,\ast})\times H_1(S^1)\to H_1(S_{g,\ast})\]
induces a map
\begin{equation}\label{eq:J}
J\colon M_f\to \Jac(S_g)
\end{equation}
 unique up to homotopy. Now $\tau_1([f])$ is defined as the image of the fundamental class
$[M_f]\in H_3(M_f,\Q)$ under the induced map $J_*\colon H_3(M_f,\Q)\to
H_3(\Jac(S_g),\Q)\approx \bwedge^3 H$.

For the definition of $\tau_i$ for all $i\geq 1$, let $\Tor_{g,\ast}$ denote the
\emph{Torelli space} of Riemann surfaces diffeomorphic to $S_g$
endowed with a homology marking and a marked point; that is,
$\Tor_{g,\ast}$ is the quotient of the Teichm\"{u}ller space
$\Teich_{g,\ast}$ by the (free) action of $\I_{g,\ast}$.  As
$\Teich_{g,\ast}$ is contractible, we have that
$H_i(\I_{g,\ast})\approx H_i(\Tor_{g,\ast})$.  Let
\begin{equation}
  \label{eq:univbundle}
  S_{g,\ast}\to \Tor^\ast_{g,\ast}\overset{\pi}{\to}\Tor_{g,\ast}
\end{equation}
denote the universal $S_{g,\ast}$--bundle over $\Tor_{g,\ast}$.  There is also a
universal bundle
\begin{equation}
  \label{eq:jacbundle}
  \Jac(S_{g}) 
  \to \Tor^{\Jac}_{g,\ast}
  \to \Tor_{g,\ast}
\end{equation} 
of Jacobians, and the Abel--Jacobi map $j$ globalizes to give a (holomorphic) map
\begin{equation}
\label{eq:globalAJ}
J\colon \Tor^*_{g,\ast} \to\Tor^{\Jac}_{g,\ast}.
\end{equation}
 Note that the
zero element in each fiber gives a section of the bundle
(\ref{eq:jacbundle}), and that $\I_{g,*}$ acts trivially on the fiber
$\Jac(S_{g})$. These are the only obstructions to triviality for a
torus bundle, so the bundle $\Tor^{\Jac}_{g,\ast} \to \Tor_{g,\ast}$
is topologically trivial.  Let $p\colon \Tor^{\Jac}_{g,\ast}\to
\Jac(S_{g})$ denote projection to the torus factor.

We can now define $\tau_i(\sigma)$ for an $i$--cycle $\sigma$ in
$\Tor_{g,\ast}$, where $i\leq 2g-2$.  The inverse image $\pi^{-1}(\sigma)$ in
$\Tor^\ast_{g,\ast}$ gives an $(i+2)$--cycle (this operation is
sometimes called the \emph{Gysin homomorphism}), and we can take the
image of this $(i+2)$--cycle under the composition $p\circ J$, giving
us an element $\tau_i(\sigma)\in H_{i+2}(\Jac(S_{g}),\Q)\approx
\bwedge^{i+2}H$.

The images of the maps $\tau_0$ and $\tau_1$ are worked out explicitly in Section
\ref{section:examples} below.  In particular, as proved by Hain in
\cite{Ha1} (see also Proposition \ref{prop:tau1} below), the map
$\tau_1$ agrees with the original, purely algebraic definition of the
\emph{Johnson homomorphism}, which plays a central role in the study
of $\I_{g,\ast}$.  In a series of papers, Johnson proved that $\tau_1$
is, modulo torsion, an isomorphism (see \cite{Jo} for a summary of
this work).  In his 1983 paper \cite{Jo}, Johnson constructed $\tau_i$
as above, and as Question C he asked if $\tau_i$ is a rational
isomorphism for all $i\geq 1$.  In \cite{Ha1} Hain used continuous
cohomology and representation theory to prove that $\tau_2$ is not
injective; it seems that Hain's method cannot be extended to the case
when $i\geq 3$.  In this paper we develop a method for concretely
computing the values of the $\tau_i$.  Our first main result answers
Johnson's question negatively in degrees $2\leq i<g$.

\begin{theorem}\label{theorem:noninjective}
The map $\tau_i$ is not injective for any $2\leq i< g$.
\end{theorem}

\para{Remark}  The map $\tau_i$ can be defined on integral homology, with target $\bwedge^{i+2}H_\Z$.  Since the target is free abelian, and since the elements we construct in the 
kernel of $\tau_i$ are integral classes, Theorem~\ref{theorem:noninjective} implies that our classes also lie in the kernel of this integral version of $\tau_i$.

\bigskip
We will find a number of sources for nontrivial cycles in
$\ker \tau_i$.  One source will be certain ``abelian cycles'' coming from
bounding pair maps (see below).  These cycles are determined by
certain collections of simple closed curves.  The (non)vanishing of
$\tau_i$ on such cycles will depend on the topological configuration
of the collection of curves, namely whether or not they are ``truly
nested'' (see Definition~\ref{def:trulynested}).  The nontriviality of
cycles in the kernel of $\tau_i$ is detected by combining certain
operations in the homology of Torelli groups with other $\tau_j$ for
$j\neq i$.  We remark that Bestvina--Bux--Margalit \cite{BBM} found
nontrivial elements of $H_{3g-3}(\I_{g,\ast},\Q)$; there is no $\tau_i$ defined in this dimension 
since $3g-3>2g-2$.

In the positive direction of Johnson's question, we show that the
$\tau_i$ detect nontrivial classes in each dimension; in particular we
prove that $\tau_2$ is surjective.  Our general theorem in this
direction is most simply stated in the language of symplectic
representation theory.  From the standard exact sequence
\[1\to\I_{g,\ast}\to\Mod_{g,\ast}\to\Sp(2g,\Z)\to 1,\] the conjugation
action of $\Mod_{g,\ast}$ on $\I_{g,\ast}$ descends to an action of
$\Sp(2g,\Z)$ by outer automorphisms, which gives $H_i(\I_{g,1},\Q)$
the structure of an $\Sp(2g,\Z)$--module.  The construction of the homomorphism
$\tau_i$ shows that it is $\Sp(2g,\Z)$--equivariant.

\para{Irreducibility remark}  Let $V$ be an irreducible
$\Sp(2g,\Q)$--representation.  It follows from Proposition 3.2 of
\cite{Bo} that $V$ is an irreducible $\Sp(2g,\Z)$--module (this is
close to the statement of the Borel Density Theorem in this case).
Henceforth we will not make the distinction of irreducibility over
$\Q$ versus irreducibility over $\Z$.  \bigskip

The algebraic irreducible representations of $\Sp(2g,\Q)$ are classified by
their highest weight vectors (a good reference is \cite{FH}). Choose a set
$\lambda_1,\ldots,\lambda_g$ of fundamental weights for
$\Sp(2g,\Q)$. A highest weight vector is a linear combination
$\lambda=\sum c_i\lambda_i$, where the coefficients are nonnegative
integers. We denote the irreducible representation of $\Sp(2g,\Q)$
with highest weight vector $\lambda$ by $V(\lambda)$. For example,
$V(\lambda_i)$ is the kernel of the contraction
$C_i:\bwedge^iH\to\bwedge^{i-2}H$ defined by:
\begin{equation}
\label{eq:contraction}
C_i(x_1\wedge\cdots\wedge x_i)=\sum_{j<k}(-1)^{j+k+1}\omega(x_j,x_k) x_1\wedge \cdots\wedge \widehat{x}_j\wedge\cdots\wedge \widehat{x}_k\wedge \cdots \wedge x_i
\end{equation}
 The $\Sp(2g,\Q)$--module $\bwedge^k
H$ for $k\leq g$ decomposes into irreducible representations as
\[\bwedge^k H\approx V(\lambda_k)\oplus
V(\lambda_{k-2})\oplus \cdots\oplus V(\lambda_\epsilon)\] where
$\epsilon=0$ or $1$ depending on whether $k$ is even or odd.  Our
second main result is the following.

\begin{theorem}
  \label{theorem:surjective}
  Suppose $g\geq 2$. Then for $1\leq i\leq g-2$, we have
  \begin{equation}\label{eq:repsall}\tau_i(H_i(\I_{g,\ast},\Q))\supseteq
    V(\lambda_{i+2}) \oplus V(\lambda_{i}).
  \end{equation}
  In addition, for $1\leq i\leq g$ and $i$ even, we
  have \begin{equation}\label{eq:repseven}\tau_i(H_i(\I_{g,\ast},\Q))
    \supseteq V(\lambda_{i-2}).
  \end{equation}
\end{theorem}
\noindent For $i=g-1$, the term $V(\lambda_{i+2})$ in
\eqref{eq:repsall} is not meaningful, but the proof of
Theorem~\ref{theorem:surjective} will show that
$\tau_{g-1}(H_{g-1}(\I_{g,\ast},\Q))$ contains $V(\lambda_{g-1})$.
 
\bigskip Since $\bwedge^4H=V(\lambda_4)\oplus V(\lambda_2)\oplus
V(\lambda_0)$, we have the following.

\begin{corollary}\label{cor:tau2}
  Let $g\geq 2$.  Then $\tau_2$ is surjective.
\end{corollary}

We wish to point out that Morita announced in \cite{Mo2} that a map closely 
related to $\tau_2$ is surjective.  As another corollary of Theorem
\ref{theorem:surjective} we deduce the following.

\begin{corollary}
  Let $g\geq 2$.  Then $H_i(\I_{g,*},\Q)$ is nonzero for each $1\leq
  i<g$. When $g$ is even, $H_g(\I_{g,*},\Q)$ is also nonzero.
\end{corollary}

Theorem~\ref{theorem:surjective} also provides evidence for a
``homological stability'' conjecture for the Torelli group, which we
now outline.

\para{Stable classes} The nontrivial classes we construct above
are stable.  In order to explain this we need to extend our picture to
surfaces with boundary.  Let $\I_{g,1}$ denote the group of homotopy
classes of homeomorphisms of the compact genus $g\geq 2$ surface
$S_{g,1}$ with one boundary component, acting trivially on
$H_1(S_{g,1},\Z)$.  Here both the homeomorphisms and homotopies are taken
to be the identity on $\partial S_{g,1}$.

The map $S_{g,1}\to S_{g,*}$ that identifies $\partial S_{g,1}$ to a
single (marked) point gives a homomorphism $\nu\colon \I_{g,1}\to \I_{g,\ast}$
whose kernel is the cyclic group generated by the Dehn twist about
$\partial S_{g,1}$.  We define a homomorphism
\[\widehat{\tau_i}\colon H_i(\I_{g,1},\Q)\to \bwedge^{i+2}H\]
by composing $\tau_i$ with the map on homology induced by $\nu$. From the proof of
Theorem~\ref{theorem:surjective}, we immediately obtain, for $1\leq
i\leq g-2$, that \[\widehat{\tau_i}(H_i(\I_{g,1},\Q))\supseteq
V(\lambda_{i+2})\oplus V(\lambda_i).\] Now, the natural inclusion
$S_{g,1}\hookrightarrow S_{g+1,1}$ induces a natural inclusion
$\I_{g,1}\hookrightarrow \I_{g+1,1}$.  We can thus form the direct 
limit
\[\I_{\infty,1}\coloneq \lim_{g\rightarrow\infty}\I_{g,1},\]
called the \emph{stable Torelli group}. It is easy to see from the
definitions that the following diagram is commutative, where
$H_g=H_1(S_{g,1},\Q)$ and $H_{g+1}=H_1(S_{g+1,1},\Q)$:
\[\xymatrix{
H_i(\I_{g,1},\Q)\ar[r]\ar_{\widehat{\tau_i}}[d]&H_i(\I_{g+1,1},\Q)\ar^{\widehat{\tau_i}}[d]\\
\bwedge^i H_g\ar[r]&\bwedge^i H_{g+1}
}\]
It follows that each nontrivial class in $H_i(\I_{g,1},\Q)$
constructed above is \emph{stable}, in that its image in
$H_i(\I_{g+k,1},\Q)$ is nontrivial for each $k\geq 0$.  As homology
preserves direct limits, and since $\dim V(\lambda_i)\to \infty$ as
$g\to \infty$, we have the following corollary.

\begin{corollary}
\label{corollary:inftor}
  For each $i\geq 1$, the vector space $H_i(\I_{\infty,1},\Q)$ is
  infinite-dimensional.
\end{corollary}

This greatly contrasts with the situation for the stable mapping class
group, whose odd-dimensional homology vanishes, and whose
even-dimensional homology has finite rank (see \cite{MW}).  

\para{A stability conjecture for $H_\ast(\I_{g,1},\Q)$} The stability
of the homology classes we construct, together with the presence of
nontrivial $\Sp(2g,\Z)$--modules in $H_i(\I_{g,1},\Q)$, shows that the
classical kind of homological stability, satisfied for example by
$\GL_n, \Out (F_n)$, and the mapping class group, does not hold for
the Torelli group.  However, our results provide evidence for a new
kind of ``representation-theoretic stability'', which we now describe.

We begin with the simplest, quickest-to-state form of our conjecture.
When we want to emphasize the group that acts, we will denote by
$V(\lambda)_{2g}$ the irreducible
$\Sp(2g,\Z)$--representation with highest weight vector $\lambda$.

\begin{conjecture}[Representation stability, I]
\label{conjecture:repstab1}
The homology of the Torelli group is \emph{representation stable} with
respect to $g$: for each $i\geq 1$ and each $g$ sufficiently large
(depending on $i$), we have that the $\Sp(2g,\Z)$--module
$H_i(\I_{g,1},\Q)$ contains the representation $V(\lambda)_{2g}$ with
some multiplicity $0\leq m\leq \infty$ if and only if for each $h\geq
g$ the $\Sp(2h,\Z)$--module $H_i(\I_{h,1},\Q)$ contains the
representation $V(\lambda)_{2h}$ with multiplicity $m$, and similarly
for $\I_{g,\ast}$ and $\I_g$.
\end{conjecture}

Applying a result of Kawazumi--Morita \cite[Theorem 5.5]{KM}, it can
be deduced that the truth of this conjecture for $\I_{g,\ast}$ is
equivalent to the truth of the conjecture for $\I_g$. We expect that
the conjecture for $\I_{g,1}$ is similarly equivalent.

Morita has conjectured \cite[Conjecture 3.4]{Mo1} that the
$\Sp$--invariant stable cohomology of $\I_{g,1}$ is generated by the
even Miller--Morita--Mumford classes.  Morita's Conjecture would immediately
imply the special case of Conjecture \ref{conjecture:repstab1} when
$V(\lambda)$ is the trivial representation.
 
We would like to refine Conjecture \ref{conjecture:repstab1} by giving
a more direct comparison of the homology of different Torelli groups.
Of course we cannot ask for an isomorphism of $H_i(\I_{g,1},\Q)$ and
$H_i(\I_{h,1},\Q)$ as modules since the first is an
$\Sp(2g,\Z)$--module and the second is an $\Sp(2h,\Z)$--module.
However, there are meaningful injectivity and surjectivity statements
one can ask for, as we will see in Conjecture
\ref{conjecture:repstab2} below.

Our main conjecture makes predictions about the finite-dimensional
part of $H_i(\I_{g,1},\Q)$.  We define the \emph{finite-dimensional
  homology} $H_i(\I_{g,1},\Q)^{\fd}$ to be the subspace of
$H_i(\I_{g,1},\Q)$ consisting of those vectors whose
$\Sp(2g,\Z)$--orbit spans a finite-dimensional vector space.

\begin{conjecture}[Representation stability, II]
\label{conjecture:repstab2}
For each $i\geq 1$ and each $g$ sufficiently large (depending on $i$),
the following hold: \medskip

\noindent {\bf Finite-dimensionality: }The natural map $i_\ast\colon
H_i(\I_{g,1},\Q)^{\fd}\to H_i(\I_{g+1,1},\Q)$ induced by the inclusion
$i\colon\I_{g,1}\to\I_{g+1,1}$ has image contained in $H_i(\I_{g+1,1},\Q)^{\fd}$.

\medskip
\noindent {\bf Injectivity: } The natural map
$i_\ast\colon H_i(\I_{g,1},\Q)^{\fd}\to H_i(\I_{g+1,1},\Q)$  is injective.

\medskip
\noindent {\bf Surjectivity: }The span of the $\Sp(2g+2,\Z)$--orbit of
$i_\ast(H_i(\I_{g,1},\Q)^{\fd})$ equals all of
$H_i(\I_{g+1},\Q)^{\fd}$.

\medskip
\noindent {\bf Rationality:} Every irreducible
$\Sp(2g,\Z)$--subrepresentation in $H_i(\I_{g,1},\Q)^{\fd}$ is the
restriction of an irreducible $\Sp(2g,\Q)$--representation.

\medskip
\noindent {\bf Type preserving: }For any representation
$V(\lambda)_{2g}\subset H_i(\I_{g,1},\Q)^{\fd}$, the span of the
$\Sp(2g+2,\Z)$--orbit of $V(\lambda)_{2g}$ is isomorphic to
$V(\lambda)_{2g+2}$.
\end{conjecture}

\para{Remarks}
\begin{enumerate}
\item A form of the Margulis Superrigidity Theorem (see \cite{Ma},
  Theorem VIII.B) gives that any finite-dimensional representation
  (over $\C$) of $\Sp(2g,\Z)$ either (virtually) extends to a (rational)
  representation of $\Sp(2g,\R)$ or factors through a finite
  group\footnote{One can also use the solution to the congruence
    subgroup property for $\Sp(2g,\Z), g>1$ here; see \cite{BMS}.}.
  The ``rationality'' statement of Conjecture
  \ref{conjecture:repstab2} is meant to rule out the latter
  possibility for subrepresentations of $H_i(\I_{g,1},\Q)^{\fd}$.

\item It is possible to embed the $\Sp(2g,\Q)$--module
  $V(\lambda_i)_{2g}$ into the $\Sp(2g+2,\Q)$--module
  $V(\lambda_{i+1})_{2g+2}$ so that the $\Sp(2g+2,\Q)$--span of the
  image is all of $V(\lambda_{i+1})_{2g+2}$, and similarly for other
  pairs of irreducible representations.  The ``type preserving''
  statement in Conjecture \ref{conjecture:repstab2} is meant to rule out 
  this type of phenomenon.

\item Theorem~\ref{theorem:surjective} shows that the ``stable range''
  in Conjecture~\ref{conjecture:repstab2}, meaning the smallest $g$
  for which $H_i(\I_{g,1},\Q)^{\fd}$ stabilizes, must be at least $i$.

\item Mess \cite{Me} proved that $H_1(\I_{2,1},\Q)$ contains an
  infinite-dimensional, irreducible permutation
  $\Sp(4,\Z)$--module. Similarly, the classes in
  $H_{3g-2}(\I_{g,1},\Q)$ found by Bestvina--Bux--Margalit \cite{BBM}
  span an infinite-dimensional, permutation
  $\Sp(2g,\Z)$--module. Neither of these is ``stable'' in $g$; one
  might hope that stably, such representations do not arise, and all
  irreducible $\Sp$--submodules of $H_i(\I_{g,1},\Q)$ are
  finite-dimensional for $g\gg i$.

\item Conjecture~\ref{conjecture:repstab1} and Conjecture~\ref{conjecture:repstab2} would give an affirmative answer to
  Question 7.9 of Hain--Looijenga \cite{HL}, which asked ``Is
  $H^k(T_g)$ expressible as an $\Sp_{2g}(\Z)$--module in a way that is
  independent of $g$ if $g$ is large enough?''
\end{enumerate}

\para{Evidence}
As mentioned above, Theorem~\ref{theorem:surjective} provides evidence
in every dimension for Conjectures~\ref{conjecture:repstab1} and
\ref{conjecture:repstab2}. Both conjectures are true in dimension 1, by
Johnson's computation that $H_1(\I_{g,1})\approx
H_1(\I_{g,\ast})\approx V(\lambda_3)\oplus V(\lambda_1)$ and
$H_1(\I_g)\approx V(\lambda_3)$. Other than this, very little is known
about the homology of the Torelli group. However, in low dimensions we
have work of Hain, who found a large subspace of $H_2(\I_g)$, and
Sakasai, who found a large subspace of $H_3(\I_g)$. We describe their
methods and state their results in Section~\ref{sec:HS}; the resulting
subspaces satisfy Conjecture~\ref{conjecture:repstab2}.

Since this paper was first distributed, Boldsen--Dollerup \cite{BD} have proved that the surjectivity condition in Conjecture~\ref{conjecture:repstab2} holds for $H_2(\I_{g,1};\Q)$ as long as $g\geq 6$.\\

 In the paper \cite{CF} we situate these conjectures in the much broader framework
of a general theory of ``representation stability''.

\para{Outline of paper} In \S\ref{section:examples} we outline our general approach to computing the $\tau_i$, and explicitly work
out $\tau_0$ and $\tau_1$ as a warmup.  In \S\ref{sec:tools} we give two ways of 
computing $\tau_i$.  We first show how to compute the image under $\tau_i$ of the
``product'' of a cycle supported on a subsurface with a bounding pair
map.  We then give a vanishing result for cycles built from gluing subsurfaces
along a pair-of-pants.  We then apply these tools in order to compute
both vanishing and nonvanishing results for $\tau_i$ of abelian cycles
in $H_i(\I_{g,\ast})$. Section~\ref{sec:gysin} gives a computation of
$\tau_i$ on cycles in $H_i(\I_{g,\ast})$ that are surface bundles over
certain tori in $H_{i-2}(\I_g)$.  This computation reveals the
phenomenon of an even/odd dichotomy for the nonvanishing/vanishing of
cycles; in particular we obtain many new nonzero classes in
$H_i(\I_{g,\ast})$.  In \S\ref{sec:imageandkernel} we use all the
computations above to complete the proofs of
Theorem~\ref{theorem:noninjective} and
Theorem~\ref{theorem:surjective}. We conclude by explaining how
theorems of Hain and Sakasai give further evidence for
Conjecture~\ref{conjecture:repstab1} and
Conjecture~\ref{conjecture:repstab2}.

\para{Acknowledgements} We are extremely grateful to the referee for all of their comments and corrections. 

\section{Setup and first examples}
\label{section:examples}

In this section, we outline the framework of the computations in this paper. We then compute $\tau_0$ and $\tau_1$ as the simplest
examples of our methods. For the rest of the paper, all homology groups
are taken with coefficients in $\Q$, although the reader may just as
well imagine the coefficients are $\Z$ if preferred.

\para{Bundles representing homology classes} As mentioned in the introduction, the bundle $S_{g,\ast}\to \Tor^\ast_{g,\ast}\overset{\pi}{\to}\Tor_{g,\ast}$ described in \eqref{eq:univbundle} is the \emph{universal} genus $g$ surface bundle equipped with a section and a trivialization of its fiberwise homology. (All the surface bundles we consider in this paper will be endowed with such a trivialization, namely an identification of the homology of each fiber with $H_1(S_g)$.) This means that for any base $B$, there is a bijective correspondence between such bundles over $B$ up to isomorphism and maps $f\colon B\to \Tor_{g,\ast}$ up to homotopy; the correspondence is induced by pulling back the universal bundle along the map $f$.

Given a homology class $\sigma\in H_i(\Tor_{g,\ast})$, we say that a bundle $S_{g,\ast}\to E\to B$ and a homology class $x\in H_i(B)$ \emph{represent} $\sigma$ if the induced homology class $f_*(x)\in H_i(\Tor_{g,\ast})$ is equal to $\sigma$. It is sometimes mentally simplifying to assume that $B$ is a closed manifold, which can be done as follows.
Thom \cite[Theorem II.29]{Th} proved  that every homology class in a closed orientable manifold has an integral multiple which can be represented by (the fundamental class of) a closed submanifold. This can be strengthened to show that every homology class in any CW complex has an odd integral multiple which can be represented by a closed submanifold,
see e.g. Conner \cite[Corollary 15.3]{Co}. Thus every homology class $\sigma\in H_i(\Tor_{g,\ast})$ has a multiple represented by the fundamental class $[B]\in H_i(B)$ for a bundle $S_{g,\ast}\to E^{i+2}\to B^i$ of closed manifolds. Although this assumption is not logically necessary for our arguments, we let it influence us by often referring to the representing homology class as $[B]\in H_i(B)$.

\para{Parametrized Abel--Jacobi maps}
Given a bundle $S_{g,\ast}\to E\to B$ and homology class $[B]\in H_i(B)$ representing $\sigma\in H_i(\Tor_{g,\ast})\approx H_i(\I_{g,\ast})$, we can use the bundle $E\to B$ to compute the Johnson invariant $\tau_i(\sigma)$, as follows. The globalized Abel--Jacobi map \eqref{eq:globalAJ} restricts to a map $J_E\colon E\to T^{2g}$ defined up to homotopy. We remark that the target should be thought out of not just as a torus $T^{2g}$, but as a
  $K(H_\Z,1)$, so that choosing a basis for $H_\Z$ gives corresponding
  coordinates on $T^{2g}$.

We call the map $J_E\colon E\to T^{2g}$ a \emph{parametrized Abel--Jacobi map}, since on each fiber
  it restricts to a map homotopic to the classical Abel--Jacobi map. Since $T^{2g}$ is aspherical, it is determined (up to
  homotopy) by the induced map on fundamental group, which is determined by the following two properties:
  \begin{enumerate}
  \item On the fiber $S_g$ the map $J_E$ induces the abelianization
    $\pi_1(S_g)\to H_\Z$.
  \item On the image of the section $B\to E$ the map $J_E$ is constant.
  \end{enumerate}

In this situation, the preimage of $[B]$ in $E$ is a class $[E]\in H_{i+2}(E)$. Then the Johnson invariant $\tau_i$ can be computed as follows:
\begin{equation}
\label{eq:computetaui}
\tau_i(\sigma)=(J_E)_*[E]\in H_{i+2}(T^{2g})\approx \bwedge^{i+2} H
\end{equation} The key to our computations in this paper is to find convenient models for $E\to B$ and for the parametrized Abel--Jacobi map $J_E$ so that $(J_E)_*[E]$ can be calculated explicitly.

\para{Computing \boldmath$\tau_0$} The intersection form on $H=H_1(S)$
can be represented by an element $\omega\in \bwedge^2 H$; if
$a_1,b_1,\ldots,a_g,b_g$ is a symplectic basis, we
have \[\omega=a_1\wedge b_1+\cdots+a_g\wedge b_g.\] Since $\tau_0$ is
a map from $H_0(\I_{g,\ast})\approx \Q$, it is determined by the image
of the generator.

\begin{proposition}\label{prop:tau0}
  The image of the generator under $\tau_0\colon
  H_0(\I_{g,\ast})\to \bwedge^2 H$ is $\omega$.
\end{proposition}
\begin{proof}
  Since the generator of $H_0(\I_{g,\ast})$ is induced by the
  inclusion of a point, we see that the image of $\tau_0$ is equal to
 $j_\ast[\Sigma_g] \in
  H_2(T^{2g})=\bwedge^2H$, the image of the fundamental class $[\Sigma_g]$ under the Abel--Jacobi map, which we now compute.
    
  \begin{figure}[h]
    \label{fig:wedge}
    \centering
      \includegraphics[width=120mm]{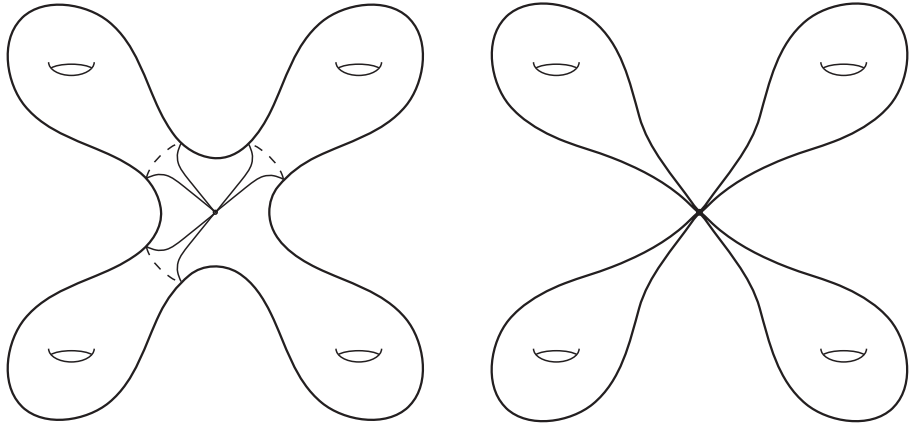}
    \caption{The surface $S_g$ and its quotient $Y=\bigvee T_i$.}
  \end{figure}
  We begin by giving an explicit construction of a map $j$ homotopic to the Abel--Jacobi map  that will be
  useful for our purposes. We will sometimes refer to such a map $j$ as \emph{an Abel--Jacobi map}, since it is uniquely defined only up to homotopy. First, let $Y=\bigvee_{i=1}^g T_i$ be
  the wedge of 2--dimensional tori. There is a natural quotient map
  $S_g\to Y$, obtained for example by collapsing a graph as in
  Figure~\figwedge.   In any torus, specifying $k$ distinct coordinates determines a
  $k$--dimensional subspace homeomorphic to a torus $T^k$. The
  coordinates of $T^{2g}$ are labeled by the symplectic basis
  $a_1,b_1,\ldots,a_g,b_g$. Identify the $i$th torus $T_i$ with the
  torus $T^2\subset T^{2g}$ determined by the $a_i$ and $b_i$
  coordinates. These identifications agree at the origins of the tori
  $T_i$, and thus induce an inclusion $Y=\bigvee T_i\hookrightarrow
  T^{2g}$. The composition  $j\colon S_g\to Y\to T^{2g}$
  is homotopic to the Abel--Jacobi map; to see this, it is enough to
  observe that the generators $a_i$ and $b_i$ of $\pi_1(S_g)$ are
  taken to the corresponding elements of $\pi_1(T^{2g})=H_\Z$.

  Finally, we must find $j_*[S_g]\in H_2(T^{2g})$.  Under the quotient map $S_g\to Y= \bigvee
  T_i$, the fundamental class $[S_g]\in H_2(S_g)$ is sent to $\sum
  [T_i]\in H_2(Y)$. Then $T_i$ is included as the torus determined by
  the $a_i$ and $b_i$ coordinates; under the natural isomorphism
  $H_k(T^{2g})\approx \bwedge^k H$, this torus represents
  $a_i\wedge b_i$. Thus we have
  \[j_\ast\colon [S_g]\mapsto \sum [T_i]\mapsto \sum a_i\wedge
  b_i=\omega\] as claimed.
\end{proof}

\para{Computing \boldmath$\tau_1$}
In \cite{JoA}, Johnson used the action of $\I_{g,\ast}$ on the second
universal $2$--step nilpotent quotient of $\pi_1(S_{g,\ast})$ to
define in a purely algebraic way an $\Sp(2g,\Z)$--equivariant
homomorphism $\tauJ:\I_{g,\ast}\to\bwedge^3H$ which is now called the
\emph{Johnson homomorphism}.

Recall that a \emph{bounding pair map} in $\I_{g,\ast}$ is a
composition of two Dehn twists $T_\alpha T_\beta^{-1}$, where $\alpha$
and $\beta$ are nonhomotopic, homologous, disjoint nonseparating
simple closed curves.  For any bounding pair map, up to homeomorphism of $S_g$, the curves $\alpha$ and $\beta$ are of the form depicted in Figure~\figtau{}a. Let $S'$ be the component of $S_g\setminus (\alpha\cup\beta)$ not containing the basepoint, and fix $1<k\leq g$ so that $S'$ has genus $k-1$. 
Let
$\{a_1,b_1,\ldots, a_g,b_g\}$ be a symplectic basis for $H_1(S_g)$ with the
property that $\alpha$ (oriented with $S'$ on the left) is homologous to $a_k$, and so that $\{a_1,b_1,\ldots,a_{k-1},b_{k-1},a_k\}$ descends to a basis for $H_1(S')$.
   Johnson showed
in \cite{JoA} that \begin{equation}\label{eq:Johnson}
  \tauJ(f)=(a_1\wedge b_1+\cdots+a_{k-1}\wedge b_{k-1})\wedge a_k.
\end{equation}

The following proposition was stated by Johnson in \cite{Jo} as the
motivation for investigating the maps $\tau_i$. Hain gave a proof in
\cite{Ha1} using the work of Sullivan and the cup product structure on
the cohomology of mapping tori. Our proof is elementary, and more
importantly, it can be generalized to higher-dimensional cycles. 
Indeed, the ideas introduced in this proof will
appear throughout Sections~\ref{sec:tools} and \ref{sec:gysin}.

\begin{proposition}[Johnson, Hain \cite{Ha1}]
 \label{prop:tau1}
  The map $\tau_1\colon H_1(\I_{g,*})\to \bwedge^3 H$ coincides with
  the Johnson homomorphism $\tau_J$.
\end{proposition}

Proposition~\ref{prop:tau1} can be
thought of as a ``parametrized'' version of the proof of
Proposition~\ref{prop:tau0} above.  

\begin{proof}
  Building on work of Birman \cite{Bi} and Powell \cite{Po}, Johnson proved in \cite{JoB}
 that $\I_{g,*}$ is generated by bounding
  pair maps for $g\geq 3$; see Hatcher--Margalit \cite{HM} for a modern proof. (For $g=2$ separating twists are also necessary; however $\tau_J$ is known to vanish on separating twists, and $\tau_1$ vanishes on separating twists by Proposition~\ref{prop:septwist}.) Thus it suffices to check that $\tau_1$ coincides on
  bounding pair maps with Johnson's map $\tauJ$.
  
  \begin{figure}[h]
    \label{fig:tau1}
    \centering
      \includegraphics[width=115mm]{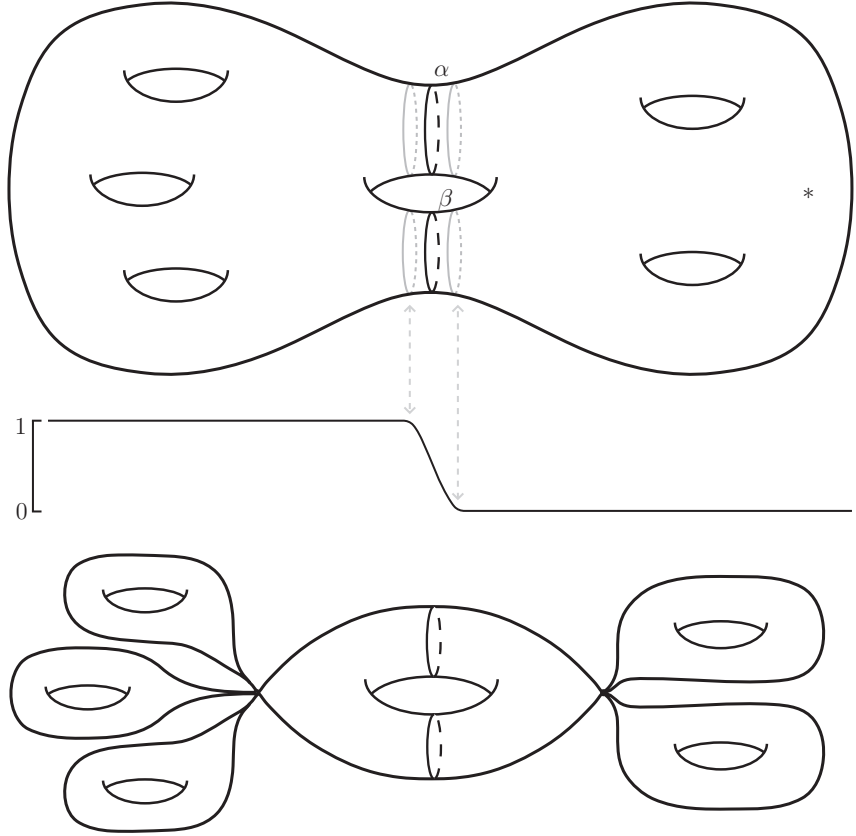}
    \caption{\textbf{a.} The surface $S_{g,*}$ and the bounding pair
      $f$. \textbf{b.} The $a_k$ component of $\delta$.\newline \textbf{c.}
    The quotient $Y=\bigcup T_i$. The torus $T_k$ is in the middle, the tori $T_i$ for $i<k$ are on the left, and the tori $T_i$ for $i>k$ are on the right.}
  \end{figure}  
  
  To compute $\tau_1(f)$, we first find a bundle $S_{g,*}\to E\to S^1$
  representing $[f]\in H_1(\I_{g,*})$. The natural choice is the
  mapping torus $S_{g,*}\to M_f\to S^1$, which can be defined as the
  quotient
  \[M_f=S_g\times [0,1]/(f(p),0)\sim (p,1).\] The image of the
  basepoint $\ast\in S_g$ in each fiber gives a section of this
  bundle.

  To describe the parametrized Abel--Jacobi map $J=J_{M_f}\colon M_f\to T^{2g}$, we will define $J$ on the cylinder $S_g\times [0,1]$ in such a way that it
  descends to $M_f=S_g\times [0,1]/\sim$. One obvious first approach
  is to define $J$ on the fiber $S_g\times \{0\}$ just by the Abel--Jacobi
  map $j$. The identification $\sim$ then forces the restriction of
  $J$ to $S_g\times \{1\}$ to be $j\circ f$. We might naively try to
  define $J$ simply by interpolating between $j$ and $j\circ f$:
  \begin{equation}
  \label{eq:firstattempt}
  J(p,t) \overset{?}{=} (1-t)\cdot j(p) + t\cdot j\circ f(p)
  \end{equation} But
  $j$ and $j\circ f$ take values in the torus $T^{2g}$, so the first
  term $(1-t)\cdot j(p)$, for example, is not well-defined. However,
  we can accomplish this idea as follows. Since $f\in\I_{g,\ast}$, the
  two maps $j\circ f$ and $j$ induce the same map on the fundamental group
  and thus are homotopic. Equivalently, their pointwise difference
  $j\circ f - j$ is homotopically trivial as a map $S_g\to T^{2g}$. We
  may thus take a lift $\delta\colon S_g\to \R^{2g}$ of $j\circ f -
  j$; that is, the unique map satisfying $\delta(\ast)=0$ and \[j\circ
  f-j=\delta\bmod{\Z^{2g}}.\] For convenience, we will take $j$ to be an Abel--Jacobi map chosen so that
  the only coordinate of $\delta$ which is nonzero is that
  corresponding to $a_k$, and in that coordinate $\delta$ is of the
  form shown in Figure~\figtau{}b. In particular $\delta(p)$ only
  depends on the ``horizontal'' coordinate of $p$ in the depiction in
  Figure~\figtau{}a.

  One way to ensure this is as follows. The twists $T_\alpha$ and
  $T_\beta$ are supported on annular neighborhoods $N_\alpha$ and
  $N_\beta$ of $\alpha$ and $\beta$ respectively. Identify these with
  $S^1\times [0,1]$ so that $T_\alpha(\theta,t)=(\theta+t,t)$ on
  $N_\alpha$ and $T_\beta^{-1}(\theta,t)=(\theta-t,t)$ on
  $N_\beta$. We define $j$ to be zero on $S_g\setminus (N_\alpha\cup N_\beta)$; on $N_\alpha$ and on
  $N_\beta$ the $a_k$ coordinate of $j$ is given by $\theta$, and all other coordinates are zero. Since
  $j=j\circ f$ outside $N_\alpha\cup N_\beta$, the function $\delta$
  is constant there. On $N_\alpha$ the $a_k$ coordinate of $j\circ
  f-j$ is given by $t$, and similarly on $N_\beta$ by $-t$. Thus
  $\delta$ has the properties claimed above.

  Now we may define $J\colon M_f\to T^{2g}$ by
  \[J(p,t)=j(p)+t\cdot\delta(p).\] Substituting $j\circ
  f-j=\delta\bmod{\Z^{2g}}$, we see that this definition realizes the idea set out in \eqref{eq:firstattempt} above. The bounding pair map $f$ does not
  factor through a wedge of tori, but it does factor through the space
  $Y=\bigcup T_i$ depicted in Figure~\figtau{}c, which is the
  union of tori $T_i$ meeting pairwise in at most 1 point. We may assume that the sympletic basis $\{a_1,b_1,\ldots,a_g,b_g\}$ was chosen so that $\{a_i,b_i\}$ descends to a basis for $H_1(T_i)$ for each $1\leq i\leq g$. It is easy
  to see that $j$ factors through $S_g\to Y$ as well, and thus so does
  $\delta$. From our explicit formula for $J$, we see that $J$ factors
  through the space \[Z\coloneq Y\times[0,1]/(f(p),0)\sim (p,1),\]
  which fibers as a bundle $Y\to Z\to S^1$. Just as $Y$ is the union
  of tori, we see that $Z$ is the union of torus bundles $T_i\to
  Z_i\to S^1$ meeting pairwise in at most a circle.The quotient $M_f\to Z$ maps the fundamental
  class $[M_f]$ to $\sum [Z_i]$. Thus it remains to understand
  $J_*[Z_i]$.

  Note that $f$ is supported on the torus $T_k$; it follows that for
  $i\neq k$ the torus bundle $T_i\to Z_i\to S^1$ is in fact a product
  $Z_i\approx T_i\times S^1$. For $i\neq k$, since $j$ and $j\circ f$
  agree on $T_i$, we have that $\delta$ is constant on $T_i$. Since
  $\delta$ is as depicted in Figure~\figtau{}b, we have that the
  $a_k$ component of $\delta$ is 1 on $T_i$ for $i<k$ and is 0 on
  $T_i$ for $i>k$. Thus when $i<k$, the restriction of $J$ to
  $Z_i\approx T_i\times S^1$ is given by
  \[J(p,t)=j(p)+(0,\ldots,0,t,0,\ldots,0).\] This is just the
  inclusion of $T^3\subset T^{2g}$ determined by the $a_i$, $b_i$, and
  $a_k$ coordinates; in particular, we have \[J_*[Z_i]=a_i\wedge
  b_i\wedge a_k \mbox{\ for \ }i<k.\] When $i>k$, we have that
  $J(p,t)=j(p)$, so the image of $J$ is contained in the 2--dimensional
  subspace determined by $a_i$ and $b_i$. Since $H_3(T^2)\subset
  H_3(T^{2g})$ is trivial, we have $J_*[Z_i]=0$ for
  $i>k$. Finally, on $T_k$ the function $\delta$ is nonconstant;
  however, the images of both $j$ and $\delta$ are contained in the
  2--dimensional subspace determined by $a_k$ and $b_k$. The same is
  thus true of the image of $J$, so $J_*[Z_k]=0$ as well. We conclude
  that
  \[J_*\colon [M_f]\mapsto \sum_{i=1}^g [Z_i]\mapsto
  \sum_{i<k}a_i\wedge b_i\wedge a_k=\tauJ(f),\] as desired.
\end{proof}

Andy Putman has pointed out that one can view the idea of this proof as ``moving the cycle represented by $M_f$ to the boundary of Torelli space'', where the computation is easier to verify; from this viewpoint, moving to the boundary of Torelli space corresponds to the degeneration of $S_g$ to the union-of-tori $Y=\bigcup T_i$.

\section{Tools for computing $\tau_i$}\label{sec:tools}
In this section we provide two of our main tools for computing $\tau_i$,
and we use them to compute $\tau_i$ on abelian cycles.  It will be
convenient for us to state our results for the case of surfaces with
boundary, namely for the map $\widehat{\tau_i}$ mentioned in the
introduction.  For simplicity of notation we will call this map
$\tau_i$ as well.

\para{Product with a bounding pair map}
Our first proposition gives a method to bootstrap up homology classes
which can be detected using $\tau_i$.  Let $S_{g,1}\hookrightarrow
S_{g+1,1}$ be the standard inclusion, inducing an inclusion
$\I_{g,1}\hookrightarrow \I_{g+1,1}$.  Let $f=T_\alpha T_\beta^{-1}$
be a bounding pair map supported in the complement of $S_{g,1}$, and
let $a$ be the common homology class of $\alpha$ and $\beta$ (oriented with $S_{g,1}$ on the left).  Then we
have a natural map \[\cdot\times f\colon H_i(\I_{g,1})\to
H_{i+1}(\I_{g+1,1})\] given by the Gysin homomorphism
$H_i(\I_{g,1})\to H_{i+1}(\I_{g,1}\times \langle f\rangle)$ followed
by the inclusion $\I_{g,1}\times \langle f\rangle\to \I_{g+1,1}$.

\begin{proposition}\label{prop:crossTstable}
  Let $f$ be as above. For any $\sigma\in H_i(\I_{g,1})$ we have \[\tau_{i+1}(\sigma\times
  f)=\tau_i(\sigma)\wedge a.\]
\end{proposition}

Note that Proposition~\ref{prop:tau1} can be deduced from Proposition~\ref{prop:tau0} by applying Proposition~\ref{prop:crossTstable}.

\begin{proof}
  Let $S_{g,1}\to E'\to B$ be a bundle with a homology class $[B]\in
  H_i(B)$ representing $\sigma\in H_i(\I_{g,1})$.  There is an associated bundle $S_{g,*}\to
  E\to B$ representing $\sigma\in H_i(\I_{g,*})$. Recall that
  by \eqref{eq:computetaui}, $\tau_i(\sigma)$ is the image of $[E]$ under the
  parametrized Abel--Jacobi map $J_E\colon E\to T^{2g}$.
  
  Similarly, there is a bundle $S_{g+1,*}\to \Ebar\to B\times
  S^1$ representing \[[B]\times [S^1]\mapsto \sigma\times f\in
  H_{i+1}(\I_{g+1,*}).\] Here $[B]\times [S^1]\in H_{i+1}(B\times S^1)$
  is the class corresponding to $[B]\otimes [S^1]\in H_i(B)\otimes
  H_1(S^1)$ under the K\"unneth formula; the preimage of $[B]\times
  [S^1]$ is a class denoted $[\Ebar]\in H_{i+3}(\Ebar)$.

  To compute $\tau_{i+1}(\sigma \times f)$, we need to explicitly
  describe the space $\Ebar$. By $S_{1,1,*}$ we mean a surface of
  genus 1 with one boundary component and a separate marked point.  We can
  glue $E'$ to the trivial bundle $S_{1,1,*}\times B$ fiberwise along 
  their common boundary component $S^1\times B$.  Now
  let \[\Ebar=(E'\cup (S_{1,1,*}\times B))\times [0,1]/\sim,\] where
  the identification is
  given by: \begin{align*}(e,0)&\sim (e,1)&\text{ for }e&\in E'\\
    \big((f(p),b),0\big)&\sim\big((p,b),1\big)&\text{ for }(p,b)&\in
    S_{1,1,*}\times B\end{align*} 
    
  Note that $\Ebar$ naturally has the
  structure of a bundle \[S_{g+1,*}\to \Ebar\to B\times S^1.\] Over
  $B\subset B\times S^1$, this bundle restricts to \[S_{g+1,*}\to
  E'\cup (S_{1,1,*}\times B)\to B,\] which represents 
  $[B]\mapsto\iota(\sigma)\in H_i(\I_{g+1,*})$. Over $S^1\subset
  B\times S^1$, it restricts to the mapping torus 
  $S_{g+1,*}\to M_f\to S^1$ of $f$, which represents $[S^1]\mapsto
  [f]\in H_1(\I_{g+1,*})$. It follows that $\Ebar\to B\times S^1$ represents
  $[B]\times [S^1]\mapsto \sigma\times f\in H_{i+1}(\I_{g+1,*})$, as
  desired.

  Now we construct the parametrized Abel--Jacobi map $J_{\Ebar}\colon \Ebar\to
  T^{2g+2}$.
  The quotient $S_{g+1,*}\to S_g\vee S_{1,*}$ induces a quotient
  $\Ebar\to Z$, where $Z$ is a bundle $S_g\vee S_{1,*}\to Z\to B\times
  S^1$. Note that $Z$ is the union of two subspaces: the first a
  bundle $S_g\to Z_1\to B\times S^1$ and the second a bundle
  $S_{1,*}\to Z_2\to B\times S^1$, meeting in a codimension 2 subspace
  homeomorphic to $B\times S^1$. By examination, we see that $Z_1$ is
  in fact simply $S_g\to E\times S^1\to B\times S^1$, and that $Z_2$
  is simply $S_{1,*}\to B\times M_f\to B\times S_1$. In particular,
  the quotient $\Ebar\to Z$ maps
  \[[\Ebar]\quad \mapsto\quad [E]\times [S^1]\ +\ [B]\times [M_f]\quad
  \in H_{i+3}(Z).\]

  We will define $J_{\Ebar}$ by defining it on the pieces $E\times
  S^1$ and $M_f\times B$ of the quotient space $Z$. Let $J_E\colon
  E\to T^{2g}$ be a parametrized Abel--Jacobi map for $E$. Let $j\colon S_{1,*}\to
  T^2$ be an Abel--Jacobi map, and as above let $\delta\colon S_{1,*}\to
  \R^2$ be the map defined by the conditions that $\delta(\ast)=0$ and
  $j\circ f-j=\delta \bmod{\Z^2}$. Assume that we have chosen a basis
  for $H_1(S_{g+1},*)$ so that $a=a_{g+1}$. We define the parametrized Abel--Jacobi
  map $J_{\Ebar}\colon Z\to T^{2g+2}=T^{2g}\times T^2$ by
  \begin{align*}
    (e,t)&\mapsto\big(J_E(e),\ (t,0)\big)
    &\text{for }(e,t)&\in E\times S^1\\
    (b,(p,t))&\mapsto\big(0,\ j(p)+t\delta(p)\big) &\text{for
    }(b,(p,t))&\in B\times M_f
  \end{align*} On the intersection $(E\times S^1) \cap (B\times M_f)$ we have
  $J_E(e)=0$, while $j(p)=0$ and $\delta(p)=(1,0)$ (this can be checked as in the proof of
  Proposition~\ref{prop:tau1}); thus the resulting map $J_{\Ebar}$ is
  well-defined.  To see that $J_{\Ebar}$ is a parametrized Abel--Jacobi map, we
  consider the restriction to a fiber and to the section. On the
  section, which is contained in $B\times M_f$, we have
  \[(b,(\ast,t))\mapsto (0,j_1(\ast)+t\delta(\ast),j_2(\ast))=0\] as
  desired. Restricted to a fiber $S_{g+1,*}$ of $\Ebar$, the map $J_{\Ebar}$
  factors through $S_g\vee S_{1,*}$. On the first component the map
  is $(J_E,0,0)$, which induces the abelianization; on the second 
  component we have $(0,J_{M_f})$, which does the same. Thus
  $J_{\Ebar}\colon \Ebar\to T^{2g+2}$ is the desired parametrized Abel--Jacobi
  map.

  It remains to compute $(J_{\Ebar})_*([E]\times [S^1])$ and
  $(J_{\Ebar})_*([B]\times [M_f])$. The restriction of $J_{\Ebar}$ to
  $E\times S^1$ is of the form $J_E\times (t,0)$; it is then immediate
  that \[(J_{\Ebar})_*([E]\times [S^1])= (J_E)_*([E])\wedge a.\] The
  image of $J_{\Ebar}$ restricted to $B\times M_f$ is contained in the
  2--dimensional subtorus determined by the last two coordinates, and
  is thus trivial in $H_{i+3}(T^{2g+2})$. It follows
  that \[\tau_{i+1}(\sigma\times
  f)=(J_{\Ebar})_*[\Ebar]=(J_E)_*[E]\wedge
  a+0=\tau_i(\sigma)\wedge a\] as desired.
\end{proof}

\para{The pair-of-pants product} Our second kind of computation of
$\tau_i$ is a vanishing result.  To state it in a general form, we make
the following definition.  There is a natural inclusion $S_{g,1}\sqcup S_{h,1} \to S_{g+h,1}$ defined by gluing two surfaces $S_{g,1}$ and $S_{h,1}$  to a
pair-of-pants $S_{0,3}$ along their boundary components, producing a
surface homeomorphic to $S_{g+h,1}$, as depicted in Figure~\figpants.
  \begin{figure}[h]
    \label{fig:pants}
    \centering
      \includegraphics[width=120mm]{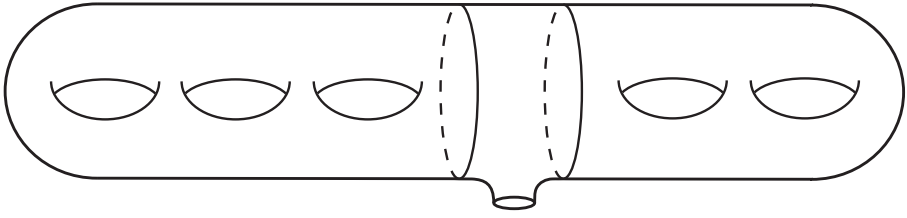}
    \caption{Gluing $S_{3,1}$ to $S_{2,1}$ to produce $S_{5,1}$.}
  \end{figure}

This inclusion
induces a map
\begin{equation}\label{eq:TorelliKuenneth}
\I_{g,1}\times \I_{h,1}\to \I_{g+h,1}
\end{equation}
The \emph{pair-of-pants product} \[H_i(\I_{g,1})\times H_j(\I_{h,1}) \to H_{i+j}(\I_{g+h,1})\] is the map obtained by composing the K\"unneth
map $H_i(\I_{g,1})\times H_j(\I_{h,1})\to H_{i+j}(\I_{g,1}\times
\I_{h,1})$ with the map on homology induced by \eqref{eq:TorelliKuenneth}.
Given $\sigma\in
H_i(\I_{g,1})$ and $\eta\in H_j(\I_{h,1})$, we denote their pair-of-pants product by
$\sigma\times\eta\in H_{i+j}(\I_{g+h,1})$.
\begin{proposition}\label{prop:pantsproducttrivial}
  Given $\sigma\in H_i(\I_{g,1})$ and $\eta\in H_j(\I_{h,1})$ with
  $i,j\geq 1$, we have\linebreak $\tau_{i+j}(\sigma\times \eta)=0$.
\end{proposition}
\begin{proof}
  Let $S_{g,1}\to E'_{\sigma}\to B_\sigma$ represent
  $[B_\sigma]\mapsto \sigma\in H_i(\I_{g,1})$, with $S_{g,*}\to
  E_{\sigma}\to B_\sigma$ representing $[B_\sigma]\mapsto \sigma\in
  H_i(\I_{g,*})$; similarly define $S_{g,1}\to E'_{\eta}\to B_\eta$
  and\linebreak $S_{g,*}\to E_{\eta}\to B_\eta$. Let $[E_\sigma]\in
  H_{i+2}(E_\sigma)$ be the preimage of $[B_\sigma]$ in $E_\sigma$,
  and similarly for $[E_\eta]\in H_{j+2}(E_\eta)$. Let the bundle
  $S_{g+h,*}\to \Ebar\to B_\sigma\times B_\eta$ represent
  $[B_\sigma]\times [B_\eta]\mapsto \sigma\times \eta\in
  H_{i+j}(\I_{g+h,*})$. The restriction of
  $\Ebar$ to $B_\sigma$ is the bundle obtained by identifying
  $S_{g,1}\to E'_\sigma\to B_\sigma$ with the trivial bundle
  $S_{h,1,*}\to S_{h,1,*}\times B_\sigma\to B_\sigma$ along their
  mutual boundary component $S^1\times B_\sigma$; a similar
  observation applies to the restriction to $B_\eta$.

  The quotient $S_{g+h,*}\to S_g\vee S_h$ induces a quotient $\Ebar\to
  Z$, where $Z$ is a bundle $S_g\vee S_h\to Z\to B_\sigma\times
  B_\eta$. The point where the two surfaces intersect gives a
  basepoint for $S_g\vee S_h$, and taking this point in each fiber
  yields a section of $Z$. Note that $Z$ is the union of two
  subspaces: the first a bundle $S_g\to Z_1\to B_\sigma\times
  B_\eta$, and the second a bundle $S_h\to Z_2\to B_\sigma\times
  B_\eta$. By inspection, we see that $Z_1$ is just $S_g\to
  E_\sigma\times B_\eta\to B_\sigma\times B_\eta$, and similarly $Z_2$
  is $S_h\to B_\sigma\times E_\eta\to B_\sigma\times B_\eta$. The
  quotient $\Ebar\to Z$ maps \[[\Ebar]\quad \mapsto\quad [E_\sigma]\times
  [B_\eta]\ +\ [B_\sigma]\times [E_\eta]\quad\in H_{i+j+2}(Z).\]

  The parametrized Abel--Jacobi map $J_{\Ebar}\colon Z\to T^{2g+2h}$ can be defined on
  $E_\sigma\times B_\eta$ by $J_{E_\sigma}\times 0$, and on
  $B_\sigma\times E_\eta$ by $0\times J_{E_\eta}$. It is easy to check
  that this is well-defined, and that it induces the appropriate map
  on fundamental group. From this formula, we see that the image under
  $J_{\Ebar}$ of the first piece $E_\sigma\times B_\eta$ is contained
  in the image of $J_{E_\sigma}$, which has dimension at most
  $i+2$. Thus $[E_\sigma]\times [B_\eta]\in H_{i+j+2}(E_\sigma\times
  B_\eta)$ is mapped to zero in $H_{i+j+2}(T^{2g+2h})$. The same
  applies to the second piece $B_\sigma\times E_\eta$, and so we
  have \[\tau_{i+j}(\sigma\times
  \eta)=(J_{\Ebar})_*[\Ebar]=(J_{\Ebar})_*([E_\sigma]\times
  [B_\eta])+(J_{\Ebar})_*([B_\sigma]\times [E_\eta])=0+0=0\] as
  desired.
\end{proof}

\para{Abelian cycles}
A collection of commuting elements $f_1,\ldots,f_d$ of a group $G$
induces a map $\Z^d\to G$; we denote the image of the fundamental
class $[\Z^d]\in H_d(\Z^d,\Q)$ in $H_d(G,\Q)$ by $\{f_1,\ldots,f_d\}$.
This is called an \emph{abelian cycle} in $H_d(G,\Q)$.  
Proposition~\ref{prop:crossTstable} and
Proposition~\ref{prop:pantsproducttrivial} can be used to compute
$\tau_i$ on certain abelian cycles. 

\begin{definition}\label{def:trulynested}
  Let $f_1,\ldots,f_k$ be a collection of bounding pair maps on
  $S_{g,*}$, with $f_i$ being the twist about $\alpha_i$ composed with
  the inverse twist about $\beta_i$.  Recall that
  the curves $\alpha_i,\beta_i$ are assumed to be nonseparating. We say that
  this collection is \emph{truly nested} if \begin{enumerate} \item
    the curves $\alpha_i$ are pairwise non-homologous, and \item after
    possibly re-ordering $\{f_i\}$, the union $\alpha_j\cup\beta_j$
    separates the basepoint from $\alpha_i\cup \beta_i$ whenever
    $i<j$.
\end{enumerate}
\end{definition}

An easy induction shows that these conditions force the curves
$\alpha_i$ and $\beta_i$ to be in one of the ``standard
configurations'', a representative example of which is given in
Figure~\figexamples{}a. Note that, by this definition, a single bounding
pair map is truly nested. For further examples, the collections
depicted in Figures~\figshriek{}a, \figabeliansurj, \figshrieksurj,
and \figdetect{} are truly nested, while the collection depicted in
Figure~\figring{} is not. We assume that any truly nested collection has been reordered so that the second condition above holds.
  \begin{figure}[h]
    \label{fig:examples}
    \centering
      \includegraphics[width=150mm]{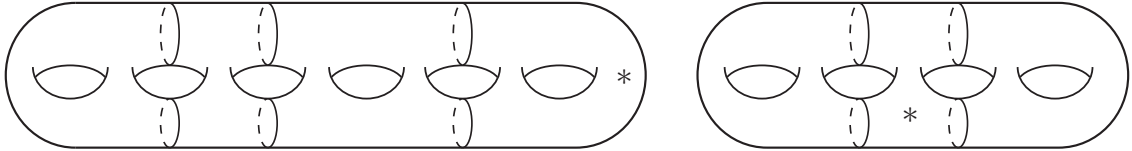}
    \caption{The first collection is truly nested; the second
      collection is not.}
  \end{figure}

  For any bounding pair $f_i$, let $c_i$ be the common homology class
  of $\alpha_i$ and $\beta_i$.  If a collection is truly nested (and
  ordered as above), then consider the ``farthest'' subsurface cut
  off, namely the component of $S_{g,*}\setminus
  (\alpha_1\cup\beta_1)$ not containing the basepoint. Choose $S_0$ to be
  a surface with one boundary component, contained in the ``farthest''
  subsurface, of maximal genus. Let $\omega_0\in \bwedge^2
  H_1(S_0)\subset \bwedge^2 H_1(S_{g,*})$ be a symplectic form for
  $H_1(S_0)$.  Note that the subspace $H_1(S_0)$ is not uniquely determined. However, the
  following theorem holds regardless of the choice of $S_0$.

\begin{theorem}\label{thm:nested}
  Any truly nested collection of bounding pair maps determines a
  nonzero abelian cycle.  More precisely, with notation as above, the
  image under $\tau_k$ of the abelian cycle $\{f_1,\ldots,f_k\}$ is
   \[\tau_k(\{f_1,\ldots,f_k\})=\omega_0\wedge c_1\wedge\cdots\wedge c_k.\]
\end{theorem}
\begin{proof}
  Let $S_0$ be as above. For $1\leq i\leq k$, sequentially choose
  $S_i$ to be a subsurface of $S_{g,1}$ with one boundary component
  and maximal genus subject to the condition that $S_i$ contains
  $\alpha_j\cup \beta_j$ for all $j\leq i$, $S_i$ contains $S_{i-1}$,
  and $S_i$ is disjoint from $\alpha_j\cup \beta_j$ for all $j>i$. The
  existence of such subsurfaces $S_i$ follows from the assumption that
  the collection is truly nested. Note that $S_k$ is the whole surface
  $S_{g,1}$.

  Let $1\in H_0(\I(S_0))$ be a generator. By
  Proposition~\ref{prop:tau0}, $\tau_0(1)=\omega_0$. For each $0\leq
  i\leq k$, the abelian cycle $\{f_1,\ldots,f_i\}$ may be considered
  as an element of $H_i(\I(S_i))$. We show by induction that
  $\tau_i(\{f_1,\ldots,f_i\})=\omega\wedge c_1\wedge\cdots\wedge
  c_i$. The abelian cycle $\{f_1,\ldots,f_{i+1}\}$ can be written as
  the cross product $\{f_1,\ldots,f_i\}\times f_{i+1}$ in the sense of
  Proposition~\ref{prop:crossTstable}, followed by the map to
  $H_{i+1}(S_{i+1})$ induced by the inclusion of the subsurface. The
  inductive step follows by applying
  Proposition~\ref{prop:crossTstable}.
\end{proof}

Conversely, we have the following.

\begin{theorem}\label{thm:non-nested}
  If $f_1,\ldots,f_k$ is a collection of commuting bounding pair maps that is
  not truly nested, then
  \[\tau_k(\{f_1,\ldots,f_k\})=0.\]
\end{theorem}

\begin{proof}
  A collection which is not truly nested must fail either the first or
  second condition in Definition~\ref{def:trulynested}.
  
  \para{Case I} We prove the following (\textit{a priori} stronger) claim: if
  the homology classes of the curves $\alpha_1,\ldots,\alpha_k$ are
  not linearly independent, then $\tau_i(\{f_1,\ldots,f_k\})=0$. Let
  $m$ be the rank of the span of the homology classes
  $c_1,\ldots,c_k$, and let $\gamma_1,\ldots,\gamma_{2k}$ be the
  $\alpha_i$ and $\beta_i$, ordered arbitrarily. We will prove below
  that it is possible to choose curves $\delta_1,\ldots,\delta_{2g}$
  with the following properties:
  \begin{enumerate}
  \item their homology classes
  $d_1,\ldots,d_{2g}$ are a symplectic basis for $H_1(S_g)$, so that
  $(d_i,d_{g+i})=1$ for $1\leq i\leq g$;
  \item for $i\leq m$ each curve
  $\delta_i$ is one of the $\gamma_j$;
  \item the span of $\langle
  d_1,\ldots,d_m\rangle$ is the span of $\langle
  c_1,\ldots,c_k\rangle$;
  \item the curve $\delta_i$ is disjoint from
  all the curves $\gamma_j$ for all $i$ except $g+1\leq i\leq g+m$.
  \end{enumerate}

  Given such a collection $\delta_1,\ldots,\delta_{2g}$, we construct
  an Abel--Jacobi map $j\colon (S_g,\ast)\to (T^{2g},0)$ supported on a
  neighborhood of the union $\delta_1\cup\cdots\cup\delta_{2g}$. One
  way to do this is to choose $1$--forms $\theta_i$ dual to $\delta_i$
  and supported in a small neighborhood. Then $j$ is defined by:
  \[j(p)=\left(\int_\ast^p \theta_{g+1}\,,\ \ldots\,,\ \int_\ast^p
    \theta_{2g}\,,\,\int_\ast^p -\theta_1\,,\ \ldots\,,\ \int_\ast^p
    -\theta_{g}\right)\] By transversality, we may assume that the
  curves $\delta_i$ intersect at most pairwise; it follows that the
  image $j(S_g)$ is contained in the $2$--skeleton of $T^{2g}$.

  Consider the bundle $S_{g,*}\to E\to T^k$ representing
  $\{f_1,\ldots,f_k\}\in H_k(\I_{g,*})$. As in the proof of
  Proposition~\ref{prop:tau1}, we may use the map $j$ to
  construct a parametrized Abel--Jacobi map $J_E\colon E\to T^{2g}$. The
  disjointness properties of $\delta_i$ imply that for each $\ell$,
  $j\circ f_\ell-j$ is nonzero only in the components determined by
  $d_1,\ldots,d_m$. It follows from the construction of $J$ that the
  image $J_E(E)$ is contained in the finite union of the tori (of
  dimension at most $m+2$) determined by the components
  $d_1,\ldots,d_m$ together with at most two other basis elements
  $d_i$ and $d_j$. This subcomplex of $T^{2g}$ has dimension $m+2$;
  since $m<k$, it follows that $(J_E)_*[E]=0$ in $H_{k+2}(T^{2g})$.

  We now show how to find such a collection $\delta_i$. We first
  find $\delta_{m+1},\ldots,\delta_{g}$ and
  $\delta_{g+m+1},\ldots,\delta_{2g}$ as follows. Consider again the
  complement $S_{g,1}\setminus\bigcup \gamma_i$. Each component of the
  complement is a surface of some genus $g_j\geq 0$; we may easily
  find $g_j$ pairs of curves on this subsurface, each pair
  intersecting in one point, and whose homology classes are a
  symplectic basis for the subspace they span. The claim is that doing
  so on each complementary subsurface yields $g-m$ such pairs. By
  collapsing to a point the genus 1 subsurface which is a regular neighborhood
  of such a pair, we may assume that each complementary subsurface has
  genus 0; to prove the claim, we need to prove that $m=g$ under this
  assumption. Consider the functionals $H_1(S_{g,1})\to \Q$ given by
  intersection with each of the $\gamma_i$. The space of functionals
  spanned by this collection has rank $m$. But if the complementary
  components have genus 0, their homology is spanned by the homology
  of their boundary components. Then Mayer--Vietoris implies that the
  mutual kernel of all these functionals is generated by the boundary
  components $\gamma_i$, and thus has rank $m$. We conclude that
  $H_1(S_{g,1})$ has rank $2m$; this verifies the claim, and so we
  have $g-m$ pairs of curves, which we take as
  $\delta_{m+1},\ldots,\delta_{g}$ and
  $\delta_{g+m+1},\ldots,\delta_{2g}$. At this point it is easy to
  choose $\delta_1,\ldots,\delta_m$ and
  $\delta_g,\ldots,\delta_{g+m}$. For the former, we choose any $m$
  curves from the $\gamma_j$ whose homology classes are linearly
  independent to be $\delta_1,\ldots,\delta_m$. Now the only condition
  on the remaining curves is that their homology classes should make
  $d_1,\ldots,d_{2g}$ a symplectic basis, so we may choose
  $\delta_{g+m+1},\ldots,\delta_{2g}$ arbitrarily subject to this
  condition. This completes the proof in the first case.

  \para{Case II} Now consider the case when the second condition is violated. We explain first
  the case when no bounding pair separates the basepoint from the
  others. Consider the component $C$ of $S_{g,1}\setminus \bigcup
  \gamma_i$ which is adjacent to the boundary component. The boundary
  of $C$ consists of curves $\alpha_i$ or $\beta_j$ (plus $\partial
  S_{g,1})$, and under our assumptions it contains curves from at
  least two bounding pairs. There must be some bounding pair $f_i$ so
  that $C$ contains both $\alpha_i$ and $\beta_i$; otherwise, without
  loss of generality the boundary of $C$ would consist of
  $\alpha_1,\ldots,\alpha_j$ for some $j$, plus $\partial
  S_{g,1}$. But then the homology classes of these curves would be
  linearly dependent, and this case has already been dealt with. Thus
  $C$ contains both $\alpha_i$ and $\beta_i$ for some $i$, and so
  there is a separating curve $\gamma$ in $C$ cutting off exactly
  $\alpha_i$ and $\beta_i$. Extend this arbitrarily to a pair-of-pants
  $S_{0,3}$ contained in $C$ having both $\gamma$ and $\partial
  S_{g,1}$ as boundary components.

  Note that $S_{g,1}\setminus S_{0,3}$ has two components $S_{h,1}$
  and $S_{g-h,1}$, each of which contains at least one bounding pair.
  Relabeling, we may assume that $f_1,\ldots,f_j$ are contained in
  $S_{h,1}$ and $f_{j+1},\ldots,f_k$ are contained in $S_{g-h,1}$ for
  $0<j<k$. Then the abelian cycle $\{f_1,\ldots,f_k\}\in
  H_k(\I_{g,1})$ is obtained as the pair-of-pants product of
  $\{f_1,\ldots,f_j\}\in H_j(\I_{h,1})$ and $\{f_{j+1},\ldots,f_k\}\in
  H_{k-j}(\I_{g-h,1})$. Applying Proposition~\ref{prop:crossTstable}, we
  conclude that $\tau_k(\{f_1,\ldots,f_k\})=0$.

  In general such a configuration will be present, but not necessarily
  adjacent to the basepoint. We attempt to order the bounding pairs
  inductively as $f_k$, $f_{k-1}$, etc., so that for each $i$ the
  union $\alpha_i\cup\beta_i$ separates the basepoint from all
  bounding pairs not yet labeled. Since the collection is not truly
  nested, at some point we cannot continue this process; we are left
  with some subset $\{f_1,\ldots,f_\ell\}$ which cannot be so
  ordered. Let $S_\ell$ be a subsurface with one boundary component
  and maximal genus subject to the condition that $S_\ell$ contains
  $\alpha_i\cup \beta_i$ if $i\leq \ell$ and is disjoint from
  $\alpha_i\cup \beta_i$ for $i>\ell$. Then $\{f_1,\ldots,f_\ell\}\in
  H_\ell(\I(S_\ell))$ is as discussed in the previous two paragraphs,
  and so $\tau_\ell(\{f_1,\ldots,f_\ell\})=0$. Now just as in the
  proof of Theorem~\ref{thm:nested}, we may filter $S_{g,1}$ by nested
  subsurfaces $S_i$ for $\ell\leq i\leq k$ with $S_i$ containing
  $\alpha_j\cup\beta_j$ iff $j\leq i$. As before,
  $\{f_1,\ldots,f_{i+1}\}$ is the cross product
  $\{f_1,\ldots,f_i\}\times f_{i+1}$, so applying
  Proposition~\ref{prop:crossTstable}, we have by induction
  \[\tau_{i+1}(\{f_1,\ldots,f_{i+1}\})=\tau_i(\{f_1,\ldots,f_i\})\wedge
  c_i=0\wedge c_i=0. \qedhere\]
\end{proof}

\para{Separating twists} We can try to generalize these techniques
beyond bounding pair maps.  In general, given $f\in \I_{g,\ast}$ and
$\sigma\in H_i(\I_{g,\ast})$, we cannot form $\sigma\times f\in
H_{i+1}(\I_{g,\ast})$. However, consider the inclusion of the
centralizer $C_{\I}(f)$ into $\I_{g,\ast}$; if $\sigma$ is represented
by some $\widetilde{\sigma}\in H_i(C_\I(f))$, we can consider
$\widetilde{\sigma}\times f\in H_{i+1}(C_\I(f)\times \langle
f\rangle)$ and define its image to be $\sigma\times f\in
H_{i+1}(\I_{g,\ast})$. Of particular importance is the case when $f$
is a twist $T_\gamma$ about a separating curve. However, unlike
bounding pair maps, separating twists do not produce nontrivial
abelian cycles with respect to $\tau_i$.

\begin{proposition}\label{prop:septwist}
Let $T_\gamma$ be the Dehn twist about a separating curve $\gamma$, and let\linebreak
$\sigma\in H_i(\I_{g,*})$ be such that 
$\sigma\times T_\gamma$ is well-defined.  Then $\tau_{i+1}(\sigma\times
  T_\gamma)=0$.
\end{proposition}
\begin{proof}
  Let $S_{g,*}\to E\to B$ represent $[B]\mapsto \sigma$. By assumption
  we may assume that the classifying map factors
  through $C_\I(f)$, so the entire image of $\pi_1(B)\to \I_{g,*}$
  fixes the curve $\gamma$. Thus by fiberwise collapsing $\gamma$ to
  a point, we have the quotient $E\to Y$, where $Y$ fibers
  as \[S_{h}\vee S_{g-h}\to Y\to B\] for some $1\leq h<g$. This is the
  union of two subspaces, $S_h\to Y_1\to B$ and \linebreak
  $S_{g-h}\to Y_2\to
  B$. Since $\gamma$ is separating, we may start with an Abel--Jacobi map
  $j\colon S_{g,*}\to T^{2g}$ so that $\gamma$ is mapped to 0, so we
  can find a parametrized Abel--Jacobi map $J_E\colon E\to T^{2g}$ which factors
  through $Y$.

  The class $\sigma\times T_\gamma$ is represented by $S_{g,*}\to
  \Ebar\to B\times S^1$, where as above \linebreak 
  $\Ebar=E\times
  [0,1]/(T_\gamma(p),0)\sim (p,1)$. As above, $\Ebar$ descends to a
  quotient \linebreak $S_{h}\vee S_{g-h}\to Z\to B\times S^1$. This is the union
  of two subspaces, which are easily seen to be products $Y_1\times
  S^1$ and $Y_2\times S^1$. We may define $J_{\Ebar}\colon Z\to
  T^{2g}$ on both $Y_1\times S^1$ and $Y_2\times S^1$ by $J_E\times
  0$. Thus $J_{\Ebar}$ factors through the $(i+2)$--dimensional complex
  $Y$, and so
  \[\tau_{i+1}(\sigma\times
  T_\gamma)=(J_{\Ebar})_*[\Ebar]=0.\qedhere\]
\end{proof}

\section{The Gysin homomorphism and $\tau_i$}\label{sec:gysin}
In this section we show how the Gysin homomorphism can be used to construct nonzero cycles detectable by $\tau_i$.   To this end, consider the universal surface bundle
\[1\to S_g \to \Tor_{g,*}\overset{\pi}{\to} \Tor_g\to 1.\] We then
have the Gysin homomorphism $\pi^!\colon H_i(\I_g)\to
H_{i+2}(\I_{g,*})$; by precomposing with the map $H_i(\I_{g,*})\to H_i(\I_g)$ induced by $\pi$, we can also consider $\pi^!$ as a map $H_i(\I_{g,*})\to
H_{i+2}(\I_{g,*})$. Composing with $\tau_{i+2}$ we obtain
\[\tau_{i+2}\circ \pi^!\colon H_i(\I_{g,\ast})\to \bwedge^{i+4} H.\]

We can use this map to detect new nontrivial cycles in $H_{i+2}(\I_{g,\ast})$.
Let $\{f_1,\ldots,f_k\}$ be a truly
  nested collection of bounding pair maps with homology classes
  $c_1,\ldots,c_k$. As before, consider the component of $S_g\setminus (\alpha_1\cup \beta_1)$ not containing the basepoint (the ``farthest'' subsurface), let $S_0$ be a maximal subsurface with one boundary component, and let $\omega_0$ represent the symplectic form on $H_1(S_0)$. Similarly, consider the component of $S_g\setminus (\alpha_k\cup \beta_k)$ containing the basepoint (the ``closest'' subsurface), let $S^0$ be a maximal subsurface with one boundary component, and let $\omega^0$ represent the symplectic form on $H_1(S^0)$. The following theorem holds regardless of the choice of $S_0$ and $S^0$.

\begin{theorem}\label{thm:evenshriek}
  Let $k\geq 2$ be even, and let $\{f_1,\ldots,f_k\}$ be a truly
  nested collection of bounding pair maps with homology classes
  $c_1,\ldots,c_k$. Then with $\omega_0$ and $\omega^0$ as above,
  \[\tau_{k+2}\big( \pi^!\{f_1,\ldots,f_k\}\big)=2\cdot \omega_0\wedge\omega^0
  \wedge c_1\wedge \cdots\wedge c_k.\]
\end{theorem}

In contrast, when $k$ is odd, we have the following theorem.
\begin{theorem}\label{thm:oddshriek}
  If $k$ is odd, then $\tau_{k+2}\circ \pi^!$ is the zero map.
\end{theorem}

Before proving these theorems, we interpret $\pi^!$ in terms of
bundles as above. The composition $\Tor_{g,*}^*\to \Tor_{g,*}\to
\Tor_g$ yields (as we will show in the following two paragraphs) a fiber bundle $S_g\times S_g\to
\Tor_{g,*}^*\overset{\Pi}{\to}\Tor_g$ with associated Gysin
homomorphism $\Pi^!\colon H_i(\I_g)\to H_{i+4}(\I_{g,*}^*)$. Recall
that $J\colon \I_{g,*}^*\to H_\Z$ is the homomorphism which is the
abelianization on the fiber and trivial on the subgroup $\I_{g,*}$. By
definition, we have \[\tau_{k+2}\circ \pi^!=J_*\circ \Pi^!\colon
H_k(\I_g)\to H_{k+4}(\I_{g,*}^*)\to H_{k+4}(H_\Z)\approx \bwedge^{k+4}
H.\] This can be described explicitly in terms of bundles, as follows.

Let $S_g\to E\to B$ represent $[B]\mapsto \sigma\in H_k(\I_g)$, with
$[E]\in H_{k+2}(E)$ denoting the preimage of $[B]$. Let $S_g\to \Ebar\to E$ be
the pullback of $S_g\to E\to B$ to $E$ by the map $p\colon E\to B$,
and let $[\Ebar]\in H_{k+4}(\Ebar)$ be the preimage of $[E]$. This
pullback consists of pairs of points $(e_1,e_2)\in E\times E$ such
that $p(e_1)=p(e_2)$. Thus the ``diagonal'' consisting of pairs
$(e,e)$ gives a section $s\colon E\to \Ebar$ of the bundle $S_g\to
\Ebar\to E$.  In summary, we have the following diagram:

\[\xymatrix@-10pt@!{
  S_g\times S_g\ar@{}[r]|-*!<0pt,-2pt>{\supset\Delta=}\ar[dr]
  &S_g\ar[d]&S_g\ar[d]\\
  &\Ebar\ar[r]\ar[d]\ar^{\Pi}[dr]&E\ar^{\pi}[d]\\&E\ar[r]&B}\]
  
By composing with the map $E\to B$, we can consider $\Ebar$ as a
bundle over $B$. The fiber $F$ is a bundle-with-section of the form $S_{g,*}\to F\to S_g$.
It can be verified that the monodromy $\pi_1(S_g,\ast)\to \Mod_{g,*}$ is contained in the kernel of the natural map $\Mod_{g,*}\to \Mod_g$ (it is easy to check that this kernel is contained in $\I_{g,*}$). Indeed Birman proved (see, e.g.\ \cite[Theorem 4.6]{FM}) that this map gives an isomorphism
\begin{equation}\label{eq:pointpushing}
\pi_1(S_g,\ast)\approx \ker(\Mod_{g,*}\to \Mod_g)
\end{equation} It thus follows that as a surface bundle, $F\approx S_g\times
S_g$. The section $s$ intersects each fiber $F$ in the diagonal
$\Delta\subset S_g\times S_g$.

Since $[\Ebar]=\Pi^!([B])$, we have that $\tau_{k+2}(\pi^!(\sigma))$
is the image of $[\Ebar]$ under the parametrized Abel--Jacobi map $J_{\Ebar}\colon
\Ebar\to T^{2g}$, which can be constructed as follows.  Let $J_E\colon
E\to T^{2g}$ be a parametrized Abel--Jacobi map, and define $J_{\Ebar}\colon
\Ebar\to T^{2g}$ to be
\[J_{\Ebar}\big((e_1,e_2)\big)=J_E(e_1)-J_E(e_2).\] As above, to
verify that $J_{\Ebar}$ is the parametrized Abel--Jacobi map for $\Ebar$, we need
to check that the induced map on fundamental group is trivial when
restricted to the section $s$, and is the abelianization when
restricted to a fiber $S_g$. The former is immediate, since $s$
consists of pairs $(e,e)$. The fiber $S_g$ is the set of pairs
$\{(e,e_0)\}$ in $\Ebar$, for some fixed $e_0\in E$. The map $e\mapsto
(e,e_0)$ identifies this with the fiber of $E$ containing $e_0$. Note
that since $J_E$ is a parametrized Abel--Jacobi map for $E$, its restriction to a
fiber induces the abelianization. Since
$J_{\Ebar}\big((e,e_0)\big)=J_E(e)-J_E(e_0)$, the restriction of
$J_{\Ebar}$ to this fiber $S_g$ is the translate of $J_E$ by the
constant $-J_E(e_0)$. Thus when restricted to this fiber, $J_{\Ebar}$
is homotopic to $J_E$ and thus induces the same map on the fundamental group.

With this description in terms of bundles in hand, we can now prove
the theorems stated above.

\begin{proof}[Proof of Theorem~\ref{thm:oddshriek}]
  The bundle $S_g\times S_g\to \Ebar\to B$ admits a natural involution
  $\rho\colon \Ebar\to \Ebar$ defined by
  $\rho\big((e_1,e_2)\big)=(e_2,e_1)$. Note that $\rho$ covers the
  identity $B\to B$. Restricted to a fiber, this is just the
  transposition of coordinates $S_g\times S_g\to S_g\times S_g$.
  Since $S_g$ is even-dimensional, this homeomorphism is
  orientation-preserving, and so it fixes the fundamental class
  $[S_g\times S_g]\in H_4(S_g\times S_g)$. Thus by the naturality of
  the Gysin homomorphism (see e.g.\ \cite[Proposition 4.8(iii)]{Mo3})
  we have $\rho_*\circ \Pi^!=\Pi^!$.  Define $\nu\colon T^{2g}\to
  T^{2g}$ to be the map induced by the map $\R^{2g}\to\R^{2g}$ given by $v\mapsto -v$. Note that $\nu_*\colon H_k(T^{2g})\to
  H_k(T^{2g})$ is the identity when $k$ is even, and is minus the
  identity when $k$ is odd. From the way we constructed $J_{\Ebar}$,
  we see that
  \[J_{\Ebar}\circ \rho=\nu\circ J_{\Ebar}.\]
  But now we have 
  \[(J_{\Ebar})_*\circ \Pi^!=(J_{\Ebar})_*\circ \rho_*\circ
  \Pi^!=\nu_*\circ (J_{\Ebar})_*\circ \Pi^!\] Thus when $k$ is odd, we
  have $(J_{\Ebar})_*\circ \Pi^!=-(J_{\Ebar})_*\circ \Pi^!$, which
  implies \[\tau_{k+2}\circ \pi^!=(J_{\Ebar})_*\circ \Pi^!=0\] as
  desired.
\end{proof}

\begin{proof}[Proof of Theorem~\ref{thm:evenshriek}]
  Let $S_g\to E\to T^k$ be the bundle classifying the abelian cycle
  $\sigma=\{f_1,\ldots,f_k\}\in H_k(\I_g)$. Form as above the fiber
  product bundle $S_g\to \Ebar\to E$ representing $\pi^!(\sigma)$, and
  view it as a bundle $S_g\times S_g\to \Ebar\to T^k$. The
  parametrized Abel--Jacobi map $J_{\Ebar}\colon \Ebar\to T^{2g}$ can
  be defined as in the proof of Proposition~\ref{prop:tau1}. We
  construct $\Ebar$ as
  \[S_g\times S_g\times [0,1]^k/\sim\] where
  \[(p,q,t_1,\ldots,t_{i-1},0,t_{i+1},\ldots,t_k)\sim
  (f_i(p),f_i(q),t_1,\ldots,t_{i-1},1,t_{i+1},\ldots,t_k).\] Let
  $j\colon S_g\to T^{2g}$ be the Abel--Jacobi map, and let
  $\delta_i\colon S_g\to \R^{2g}$ be the unique map satisfying $j\circ
  f_i=j+\delta_i\bmod{\Z^{2g}}$ and $\delta_i(\ast)=0$. Then $J_{\Ebar}$
  can now be defined by
  \[J_{\Ebar}\big((p,q,t_1,\ldots,t_k)\big)=j(p)-j(q)+\sum_i
  t_i\big(\delta_i(p)-\delta_i(q)\big).\] The definition of $\delta_i$
  was chosen exactly so that this descends to a map $J_{\Ebar}\colon
  \Ebar\to T^{2g}$.

  \begin{figure}[h]
    \label{fig:shriek}
    \centering
      \includegraphics[width=150mm]{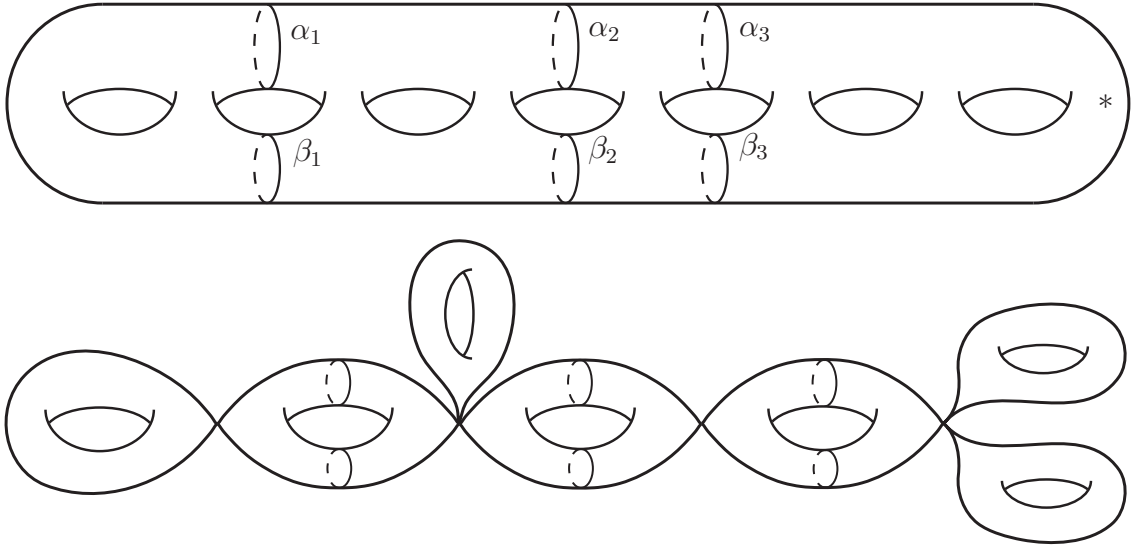}
    \caption{\textbf{a.} The bounding pair maps $f_i$. \textbf{b.} The
      quotient $Y=\bigcup T_\ell$.}
  \end{figure}
  Any truly nested collection of bounding pairs is, up to homeomorphism, of the form
  depicted in Figure~\figshriek{}a.  The maps $f_i$ factors through
  a union of 2--dimensional tori $Y=\bigcup T_\ell$, any two of which
  meet in at most one point, as depicted in Figure~\figshriek{}b. We
  have a basis $\{a_1,b_1,\ldots,a_g,b_g\}$ for $H$ so that $a_\ell$
  and $b_\ell$ span the homology of the torus $T_\ell$.  For each
  bounding pair map $f_i$, the homology class $c_i$ of its defining
  pair of curves is equal to $a_{\ell_i}$ for some $\ell_i$. As in the
  definition of a truly nested collection, we assume that
  $\ell_i<\ell_{i'}$ if $i<i'$; for simplicity, we order the $T_\ell$
  so that $T_\ell$ is separated from the basepoint by $T_{\ell_i}$ iff
  $\ell<\ell_i$.
  
  We can choose $j$ so that $j$, and thus also the
  $\delta_i$, factors through $Y$, and furthermore so that as in
  Proposition~\ref{prop:tau1}, $j$ and $j\circ f_i$ differ only in the
  component corresponding to $a_{\ell_i}$. The restriction of $j$ to
  the torus $T_\ell$ gives an identification with the linear subspace
  of $T^{2g}$ consisting of the $\langle a_\ell,b_\ell\rangle$ plane;
  parametrizing $T_\ell$ by this identification, we have that the
  restriction of $j$ to $T_\ell$ is just the inclusion of this
  subspace.

  It follows that $J_{\Ebar}$ factors through a space $Z$ which fibers 
  as a bundle\linebreak $Y\times Y\to Z\to T^k$; call the resulting map
  $J_Z\colon Z\to T^{2g}$. This bundle is the union of subspaces of
  the form $T_\ell\times T_{\ell'}\to Z_{\ell,\ell'}\to T^k$; the
  intersection of two such subspaces has codimension at least 2,
  corresponding to $T_{\ell_1}\times T_{\ell'}\cap T_{\ell_2}\times
  T_{\ell'}=\ast\times T_{\ell'}$ or to $T_{\ell_1}\times T_{\ell_1'}\cap
  T_{\ell_2}\times T_{\ell_2'}=\ast\times \ast$. It follows that the
  fundamental class $[\Ebar]\in H_{k+4}(\Ebar)$ projects to the sum of
  the fundamental classes $\sum[Z_{\ell,\ell'}]\in H_{k+4}(Z)$. Thus
  to compute $\tau_{k+2}(\pi^!\{f_1,\ldots,f_k\})
  =(J_{\Ebar})_*[\Ebar]$, it remains to understand
  $(J_Z)_*[Z_{\ell,\ell'}]$.

  Call the subspace $Z_{\ell,\ell'}$ \emph{bad} if either $\ell$ or
  $\ell'$ is equal to $\ell_i$ for some $i$; otherwise call
  $Z_{\ell,\ell'}$ \emph{good}. First, let us check that for bad
  $Z_{\ell,\ell'}$, we have $(J_Z)_*[Z_{\ell,\ell'}]=0$. For
  $(p,q,t_1,\ldots,t_k)\in Z_{\ell,\ell'}$ we have $p\in T_\ell$ and
  $q\in T_{\ell'}$; thus $j(p)$ and $j(q)$ are contained in the subspace
  determined by $\langle a_\ell, b_\ell\rangle$ and $\langle
  a_{\ell'},b_{\ell'}\rangle$ respectively. Recall that each $\delta_i$ is
  nonzero only in the coordinate corresponding to $c_i$. Our formula
  for $J_{\Ebar}((p,q,t_1,\ldots,t_k))$ thus implies that
  $J_Z(Z_{\ell,\ell'})$ is contained in the subspace determined by the
  collection $\langle a_\ell, b_\ell,
  a_{\ell'},b_{\ell'},c_1,\ldots,c_k\rangle$. However, the assumption
  that $Z_{\ell,\ell'}$ is bad implies that $a_\ell$ or $a_{\ell'}$
  coincides with some $c_i$. Thus this subspace has dimension at most
  $k+3$, and so $(J_Z)_*[Z_{\ell,\ell'}]\in H_{k+4}(T^{2g})$ must be
  zero.

  Now we consider the good pieces $Z_{\ell,\ell'}$. Since neither
  $\ell$ nor $\ell'$ is of the form $\ell_i$ for any $i$, we have that
  each map $f_i$ is the identity on $T_\ell$ and $T_{\ell'}$. It
  follows that the bundle $T_\ell\times T_{\ell'}\to Z_{\ell,\ell'}\to
  T^k$ is actually a product $Z_{\ell,\ell'}\approx T_\ell\times
  T_{\ell'}\times T^k$. Note that since $j\circ f_i=j$ on $T_{\ell}$,
  we have that each $\delta_i$ is constant on $T_{\ell}$ and is
  nonzero only in the component corresponding to $c_i$. In that
  component, we have as before that $\delta_i$ is either $1$ or $0$ on
  $T_{\ell}$, depending on whether the torus $T_{\ell}$ is cut off
  from the basepoint by $\alpha_i\cup \beta_i$ or not. Denoting
  this number by $n^i_\ell\in \{0,1\}$, we see that $n^i_\ell$ is $1$
  if $\ell<\ell_i$ and is $0$ if $\ell_i<\ell$.

  The restriction of $J_Z$ to $Z_{\ell,\ell'}\approx T_\ell\times
  T_{\ell'}\times T^k$ may now be read off from the formula for
  $J_{\Ebar}$ above. The restriction is a linear map, which can be
  described on each factor. On the first and second factors, it is the
  inclusion of $T_\ell$ as the torus determined by $\langle
  a_\ell,b_\ell\rangle$ and the inclusion of $T_{\ell'}$ as the torus
  determined by $\langle a_{\ell'},b_{\ell'}\rangle$ respectively. Let
  $\epsilon_i\in\{-1,0,1\}$ be the number $n^i_\ell-n^i_{\ell'}$. On
  the third factor $T^k$, $J_Z$ is the composition of the map $T^k\to
  T^k$ given by \begin{equation}\label{eq:flip}(t_1,\ldots,t_k)\mapsto
    (\epsilon_1 t_1,\ldots,\epsilon_k t_k)\end{equation} with the
  inclusion of $T^k$ as the torus determined by $\langle
  c_1,c_2,\ldots,c_k\rangle$. Note that the map $H_k(T^k)\to H_k(T^k)$
  induced by \eqref{eq:flip} is multiplication by
  $\epsilon_{\ell,\ell'}=\epsilon_1\cdots\epsilon_k$. From this
  description we see that the image of the fundamental class
  $[Z_{\ell,\ell'}]$ under $J_Z$ is
  \[(J_Z)_*[Z_{\ell,\ell'}]= \epsilon_{\ell,\ell'}\cdot a_\ell\wedge
  b_\ell\wedge a_{\ell'}\wedge b_{\ell'} \wedge c_1\wedge \cdots\wedge
  c_k.\] It thus remains only to understand
  $\epsilon_{\ell,\ell'}$. Since $n^i_\ell$ is $1$ if $\ell<\ell_i$
  and $0$ otherwise, we have that $\epsilon_i$ is 1 if
  $\ell<\ell_i<\ell'$, $-1$ if $\ell'<\ell_i<\ell$, and 0
  otherwise. Thus $\epsilon_{\ell,\ell'}$ is nonzero only if we have
  $\ell<\ell_1<\cdots<\ell_k<\ell'$ or
  $\ell'<\ell_1<\cdots<\ell_k<\ell$. In the former case, each
  $\epsilon_i=1$, so $\epsilon_{\ell,\ell'}=1$; in the latter, each
  $\epsilon_i=-1$, but since $k$ is even we have
  $\epsilon_{\ell,\ell'}=\epsilon_1\cdots\epsilon_k=1$ again. Note
  that if $k$ were odd, these terms would instead cancel, yielding
  another proof of Theorem~\ref{thm:oddshriek} (for the case of abelian
  cycles of bounding pairs).  Combining these cases, we conclude that
  \begin{align*}
    J_{\Ebar}\colon [\Ebar]&\mapsto \sum_{\ell,\ell'}[Z_{\ell,\ell'}]\\
    &\mapsto\sum_{\ell<\ell_1<\cdots<\ell_k<\ell'} 2\cdot a_\ell\wedge
    b_\ell\wedge a_{\ell'}\wedge b_{\ell'} \wedge c_1\wedge
    \cdots\wedge c_k\\
    &= 2\cdot \omega_0\wedge \omega^0\wedge c_1\wedge \cdots\wedge c_k
  \end{align*} and thus, as desired, \[\tau_{k+2}\big(
  \pi^!\{f_1,\ldots,f_k\}\big)=2\cdot \omega_0\wedge\omega^0 \wedge
  c_1\wedge \cdots\wedge c_k.\qedhere\]
\end{proof}

This proof is valid for $k=0$ as well, except for the last equality
above; in this case we instead have just \[\sum_{\ell<\ell'}2\cdot
a_\ell\wedge b_\ell\wedge a_{\ell'}\wedge b_{\ell'}=\omega\wedge
\omega.\] Thus if $[S_g]\in H_2(\I_{g,*})$ represents the
point-pushing subgroup as in \eqref{eq:pointpushing}, we deduce the following, which is necessary
for Corollary~\ref{cor:tau2}.
\begin{corollary} $\tau_2([S_g])=\omega\wedge \omega\in \bwedge^4 H$.
\end{corollary}

\section{The image and kernel of $\tau_i$}
\label{sec:imageandkernel}

In this section we complete the proofs of
Theorem~\ref{theorem:surjective} and
Theorem~\ref{theorem:noninjective}.

\subsection{Proof of Theorem~\ref{theorem:surjective}}

  \begin{figure}[h]
    \label{fig:abeliansurj}
    \centering
      \includegraphics[width=140mm]{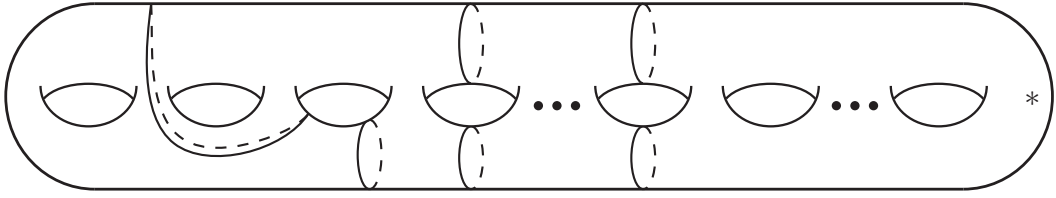}
    \caption{The collection of bounding pairs generating
      $V(\lambda_i)\oplus V(\lambda_{i+2})$.}
  \end{figure}  
  First, for all $i\leq g-2$, we show that $\tau_i(H_i(\I_{g,*}))$
  contains $V(\lambda_i)\oplus V(\lambda_{i+2})$. Consider the
  collection of bounding pairs displayed in
  Figure~\figabeliansurj. Let $\sigma\in H_i(\I_{g,*})$ be the
  associated abelian cycle; by Theorem~\ref{thm:nested} we have
  \[\tau_i(\sigma)=a_1\wedge b_1\wedge a_3\wedge \cdots
  \wedge a_{i+2}.\] Recall that there is an $\Sp$--equivariant contraction
  $C_k\colon \bwedge^kH\to \bwedge^{k-2}H$, which was defined in \eqref{eq:contraction}. We claim that
  $\tau_i(\sigma)\in\bwedge^{i+2}H$ generates $\ker C_i\circ
  C_{i+2}$ as a module. As previously noted, $\ker C_k\approx V(\lambda_k)$ and thus
  $\ker C_i\circ C_{i+2}\approx V(\lambda_i)\oplus V(\lambda_{i+2})$,
  so this will verify this case of the theorem.

  First note that $\tau_i(\sigma)$ is contained in $\ker
  C_i\circ C_{i+2}$; indeed, we have $C_{i+2}(\tau_i(\sigma))=
  a_3\wedge \cdots\wedge a_{i+2}$, which lies in $\ker C_i$. In
  particular, we see that $\tau_i(\sigma)$ is not contained in $\ker
  C_{i+2}\approx V(\lambda_{i+2})$. The element
  \[\nu\coloneq a_1\wedge (b_1+a_2)\wedge a_3\wedge \cdots \wedge a_{i+2}\]
  is clearly in the $\Sp$--orbit of $\tau_i(\sigma)$. Thus
  $\nu-\tau_i(\sigma)=a_1\wedge a_2\wedge a_3\wedge \cdots\wedge
  a_{i+2}$, which lies in $\ker C_{i+2}\approx V(\lambda_{i+2})$, is
  in the image of $\tau_i$. We conclude that the $\Sp$--span of
  $\tau_i(\sigma)$ is contained in $V(\lambda_i)\oplus
  V(\lambda_{i+2})$ and properly contains $V(\lambda_{i+2})$, and thus
  since the $V(\lambda_k)$ are irreducible, the $\Sp$--span of $\tau_i(\sigma)$
  is $V(\lambda_i)\oplus V(\lambda_{i+2})$.

  Note that for $i=g-1$, a similar collection of $g-1$ bounding pairs
  determines an abelian cycle $\sigma$ so that
  $\tau_{g-1}(\sigma)=a_1\wedge b_1\wedge a_2\wedge \cdots\wedge
  a_g\in \bwedge^{g+1}H$. The contraction
  $C_{g+1}\colon \bwedge^{g+1}H\to \bwedge^{g-1}H$ is injective \cite[Theorem 17.11]{FH}. The
  image $C_{g+1}(\tau_{g-1}(\sigma))=a_2\wedge\cdots \wedge a_g$
  clearly generates $V(\lambda_{g-1})$, and so
  $\tau_{g-1}(H_{g-1}(\I_{g,*}))\supseteq V(\lambda_{g-1})$.

  \begin{figure}[h]
    \label{fig:shrieksurj}
    \centering
      \includegraphics[width=145mm]{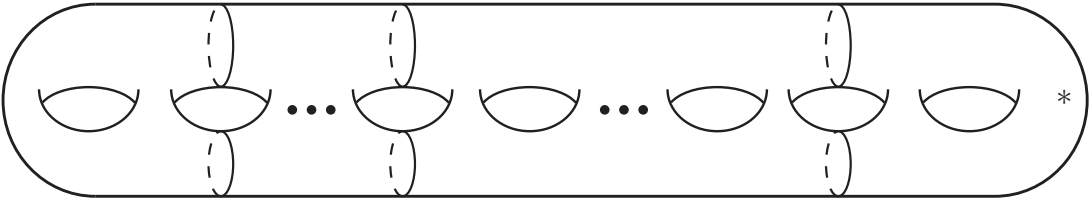}
    \caption{The $i-2$ bounding pairs used to generate
      $V(\lambda_{i-2})$.}
  \end{figure}
  We now show that when $1\leq i\leq g$ and $i$ is even,
  $\tau_i(H_i(\I_{g,*}))$ also contains $V(\lambda_{i-2})$. Take the
  collection of bounding pairs displayed in
  Figure~\figshrieksurj, and let $\sigma\in H_{i-2}(\I_{g,*})$
  be the associated abelian cycle; we will consider $\pi^!\sigma\in
  H_i(\I_{g,*})$. By
  Theorem~\ref{thm:evenshriek}, \[\tau_i(\pi^!\sigma)=2\cdot a_1\wedge
  b_1\wedge a_g\wedge b_g\wedge a_2\wedge \cdots\wedge a_{i-2}\wedge
  a_{g-1}.\] We claim that this element lies in $\ker C_{i-2}\circ
  C_i\circ C_{i+2}$, but not $\ker C_i\circ C_{i+2}$, and thus its
  $\Sp$--span contains $V(\lambda_{i-2})$ as desired. To see
  this, note that \[C_{i+2}(\tau_i(\pi^!\sigma))=2\cdot(a_1\wedge
  b_1+a_g\wedge b_g)\wedge a_2\wedge\cdots\wedge a_{i-2}\wedge
  a_{g-1},\] so \[C_i\circ C_{i+2}(\tau_i(\pi^!\sigma))=4\cdot
  a_2\wedge\cdots\wedge a_{i-2}\wedge a_{g-1},\] which lies in $\ker
  C_{i-2}$ as claimed.

\subsection{Proof of Theorem~\ref{theorem:noninjective}}

  \begin{figure}[h]
    \label{fig:ring}
    \centering
      \includegraphics[width=60mm]{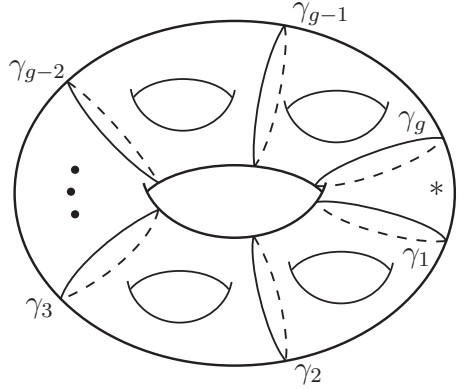}
    \caption{Nonseparating curves; each pair of adjacent curves yields
      a bounding pair $f_k$.}
  \end{figure}
  
  Consider the curves $\gamma_1,\ldots,\gamma_g$ displayed in
  Figure~\figring. For $1\leq k<g$, let  
  $f_k=T_{\gamma_{k+1}}T_{\gamma_k}^{-1}$, and for $2\leq i<g$, let
  $\sigma\in H_i(\I_{g,*})$ be the abelian cycle $\sigma\coloneq
  \{f_1,\ldots,f_i\}$. By Theorem~\ref{thm:non-nested},
  $\tau_i(\sigma)=0$ since the bounding pairs $f_k$ are not truly nested. We
  will show that $\sigma$ is nontrivial in $H_i(\I_{g,*})$, proving
  the theorem.

  Recall Johnson's map $\tauJ\colon \I_{g,*}\to \bwedge^3 H_\Z$,
  which induces \[(\tauJ)_*\colon H_i(\I_{g,*})\to H_i(\bwedge^3
  H_\Z)\approx \bwedge^i(\bwedge^3 H).\] Let $\iota\colon \Z^i\to
  \I_{g,*}$ be the inclusion of the subgroup $\langle
  f_1,\ldots,f_i\rangle$. By definition,
  \[(\tauJ)_*\sigma=(\tauJ\circ \iota)_*[\Z^i].\] From the
  identification of $H_i(\bwedge^3 H_\Z)$ with $\bwedge^i(\bwedge^3
  H)$, we have that \[(\tauJ\circ \iota)_*[\Z^i]=(\tauJ)_*[f_1]\wedge
  \cdots \wedge (\tauJ)_*[f_i].\] Let $\{a_1,b_1,\ldots,a_g,b_g\}$ be a symplectic basis for $H_1(S_g)$ so that $[\gamma_k]=a_g$ for all ${1\leq k < g}$, and so that $\{a_k,b_k\}$ gives a basis for the homology of the subsurface cut off by $\gamma_k$ and $\gamma_{k+1}$ for each $1\leq k< g$. By Johnson's computation of $\tau_J$
  (see \eqref{eq:Johnson} in Section~\ref{section:examples}
  above), $(\tauJ)_*[f_k]=a_k\wedge b_k\wedge a_g$, and thus
  \[(\tauJ)_*\sigma=\big(a_1\wedge b_1\wedge a_g\big)\wedge \cdots
  \wedge \big(a_i\wedge b_i\wedge a_g\big).\] This element is nonzero
  in $\bwedge^i(\bwedge^3 H)$ since $\{a_k\wedge b_k\wedge
  a_g\}_{k=1}^i$ is linearly independent in $\bwedge^3 H$. Thus $\tau_J(\sigma)\neq 0$ and so $\sigma\neq 0$, completing the proof of Theorem~\ref{theorem:noninjective}.

  \begin{figure}[h]
    \label{fig:detect}
    \centering
      \includegraphics[width=150mm]{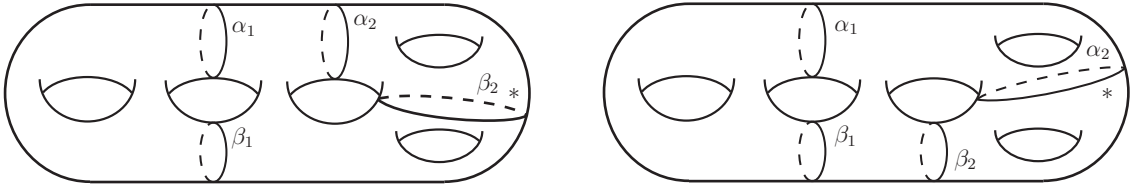}
    \caption{Homologically distinct abelian cycles not distinguishable
      by $\tau_i$.\newline \textbf{a.} The collection
      $\{f_1,f_2\}$. \textbf{b.} The collection $\{g_1,g_2\}$.}
  \end{figure}

  \begin{remark} We now give another example showing the non-injectivity of $\tau_i$. One
       notable feature of this example is that we replace $\tauJ$ in the proof above by the maps $\tau_i$
    themselves. For even $i\leq g-2$, let $\{f_k\}$ and $\{g_k\}$ be
    two truly nested collections of bounding pairs as in
    Figure~\figdetect, so that $f_k$ and $g_k$ are homologous; the
    collections cut off the same farthest subsurface; but the closest
    subsurfaces cut off by $\{f_k\}$ and $\{g_k\}$ determine different
    symplectic forms in $\bigwedge^2 H_1(S_g)$. By
    Theorem~\ref{thm:nested},
    $\tau_i(\{f_1,\ldots,f_i\})=\tau_i(\{g_1,\ldots,g_i\})$. However,
    Theorem~\ref{thm:evenshriek} shows that
    $\pi^!(\{f_1,\ldots,f_i\})$ is not equal to
    $\pi^!(\{g_1,\ldots,g_i\})$, and thus $\{f_1,\ldots,f_i\}$ is not
    equal to $\{g_1,\ldots,g_i\}$ in $H_i(\I_{g,*})$. Finally, note
    that we may choose $\{f_i\}$ and $\{g_i\}$ so that
    $(\tauJ)_*\{f_1,\ldots,f_i\}=(\tauJ)_*\{g_1,\ldots,g_i\}$, so this
    method yields new elements of $\ker \tau_i$ which cannot be
    detected by $(\tau_J)_*$.
\end{remark}

\subsection{Detecting homology using $\tau_J$}\label{sec:HS}
In general, computing the image of $(\tau_J)_*$ is very difficult; in
particular, by work of Kawazumi--Morita (see \cite[\S6.4]{Mo1}), a complete solution would
resolve the long-standing question of whether the even
Morita--Mumford--Miller classes $e_{2i} \in H^{4i}(\I_{g,*})$ are
nontrivial. However, in the lowest dimensions, the images have been
found explicitly for the related case of closed surfaces.
Considering the map
\[(\tau_J)_*\colon H_2(\I_g)\to \bwedge^2\big(\bwedge^3 H / H\big),\]
Hain \cite{Ha1} found that for $g\geq 6$, the image of $(\tau_J)_*$ is isomorphic
to
\[V(\lambda_6)\oplus V(\lambda_4)\oplus V(\lambda_2)\oplus
V(\lambda_2+\lambda_4).\] Similarly, Sakasai \cite{Sa} found that, up to
possibly a factor of $V(\lambda_1) = H$, the image of
\[(\tau_J)_*\colon H_3(\I_g)\to \bwedge^3(\bwedge^3 H / H)\]
for $g\geq 9$ is isomorphic to
\begin{align*}
&V(\lambda_5+2\lambda_2)\oplus
V(2\lambda_4+\lambda_1)\oplus
V(\lambda_6+\lambda_3)\oplus
V(\lambda_4+\lambda_3)\\\oplus
&V(\lambda_7+\lambda_2)\oplus
V(\lambda_5+\lambda_2)\oplus
V(\lambda_3+\lambda_2)\oplus
V(\lambda_6+\lambda_1)\oplus
V(\lambda_4+\lambda_1)\\\oplus
&V(\lambda_9)\oplus
V(\lambda_7)\oplus
V(\lambda_5)\oplus
V(\lambda_5)\oplus
V(\lambda_3).
\end{align*}

The stability of these decompositions is exactly the behavior
predicted by Conjecture~\ref{conjecture:repstab1}. Hain and Sakasai
also compute the decompositions for smaller $g$, but they do not
stabilize until $g\geq 6$ and $g\geq 9$ respectively.

\small

\noindent
Department of Mathematics\\ University of Chicago\\
5734 University Ave.\\
Chicago, IL 60637\\
E-mail: tchurch@math.uchicago.edu, farb@math.uchicago.edu

\end{document}